\documentclass[12pt,reqno]{amsart}
\usepackage[letterpaper,top=2cm,bottom=2cm,left=2cm,right=2cm]{geometry}
\usepackage[format=plain, font=footnotesize]{caption}
\usepackage{amsmath, amsthm, nccmath, amssymb, mathtools, hyperref%, xypic
}
\usepackage{graphicx}
\usepackage{amscd}
\hypersetup{
	colorlinks=true,       % false: boxed links; true: colored links
	linkcolor=orange,        % color of internal links
	citecolor=purple,        % color of links to bibliography
	filecolor=magenta,     % color of file links
	urlcolor=blue
}
\usepackage[shortlabels]{enumitem}
\usepackage{graphicx}

\usepackage{soul}

\usepackage[dvipsnames]{xcolor}
\usepackage{url}
\usepackage{cleveref}
\usepackage{systeme}
\usepackage{array}
\usepackage{bigints}
\usepackage{tikz-cd}
\usepackage{tabularray}
\usepackage{comment}
\usepackage{thmtools} 
\usepackage{enumitem}

\makeatletter
\renewcommand*\env@matrix[1][c]{\hskip -\arraycolsep
  \let\@ifnextchar\new@ifnextchar
  \array{*\c@MaxMatrixCols #1}}
\makeatother

\usepackage[ruled,vlined]{algorithm2e}
\usepackage{protosem}

\newcommand{\Ter}{\mathrm{Ter}}

\newcommand{\PP}{\mathbb{P}}

\newcommand{\kk}{\mathbb{K}}
\newcommand{\ZZ}{\mathbb{Z}}

\newcommand{\oo}{\mathcal{O}}
\newcommand{\h}{\mathcal{H}}

\newtheorem{theorem}{Theorem}
\numberwithin{theorem}{section}
\newtheorem{proposition}[theorem]{Proposition}
\newtheorem{lemma}[theorem]{Lemma}
\newtheorem{corollary}[theorem]{Corollary}

\theoremstyle{definition}
\newtheorem{definition}[theorem]{Definition}
\newtheorem{remark}[theorem]{Remark}
\newtheorem{example}[theorem]{Example}

\crefname{equation}{}{}
\crefname{equation}{}{}
\crefname{figure}{Figure}{Figure}
\crefname{section}{Section}{Section}
\crefname{lemma}{Lemma}{Lemma}
\crefname{proposition}{Proposition}{Proposition}
\crefname{theorem}{Theorem}{Theorem}
\crefname{corollary}{Corollary}{Corollary}
\crefname{definition}{Definition}{Definition}
\crefname{notation}{Notations}{Notation}
\crefname{remark}{Remark}{Remark}
\crefname{claim}{Claim}{Claim}
\crefname{assumption}{Assumption}{Assumption}
\crefname{example}{Example}{Example}

\newcommand{\p}{\mathbb{P}}

\newcommand{\N}{\mathbb{N}}
\newcommand{\ot}{\otimes}

\newcommand{\LL}{\mathcal{L}}

\newcommand{\Bl}{\mathop{\rm Bl}\nolimits}\newcommand{\PGL}{\mathop{\rm PGL}\nolimits}
\newcommand{\V}{\mathop{\rm V}\nolimits}
\newcommand{\SV}{\mathop{\rm SV}\nolimits}

\newcommand{\Res}{\mathop{\rm Res}\nolimits}
\newcommand{\codim}{\mathop{\rm codim}\nolimits}
\newcommand{\rd}[1]{\left\lfloor #1\right\rfloor}
\newcommand{\ru}[1]{\left\lceil #1\right\rceil}
\DeclareMathOperator{\rank}{rank}

\begin{document}

\title{Geometry of first nonempty Terracini loci}

\begin{comment}
--------------------
report number
BCSim-2022-s04
--------------------
\end{comment}

\author[F. Galuppi]{Francesco Galuppi}
\address{Faculty of Mathematics, Informatics, and Mechanics, University of Warsaw, Banacha 2, 02-097 Warsaw, Poland}
\email{galuppi@mimuw.edu.pl (ORCID 0000-0001-5630-5389)}

\author[P. Santarsiero]{Pierpaola Santarsiero}
\address{Institute of Mathematics, Osnabrück University, Albrechtstr. 28a, 49074 Osnabrück, Germany
 }
\email{pierpaola.santarsiero@uni-osnabrueck.de (ORCID 0000-0003-1322-8752)}
\author[D. Torrance]{Douglas A. Torrance}
\address{Department of Mathematical Sciences, Piedmont University, Demorest, Georgia 30535, USA}
\email{dtorrance@piedmont.edu (ORCID 0000-0003-3999-4973)}

\author[E. Turatti]{Ettore Teixeira Turatti}
\address{Department of Mathematics and Statistics, UiT The Arctic University of Norway, Forskningsparken 1 B 485, Troms\o, Norway}
\email{ettore.t.turatti@uit.no (ORCID 0000-0003-4953-3994)}

\maketitle
\begin{abstract} After a few results on curves, we characterize the smallest nonempty Terracini loci of Veronese and Segre-Veronese varieties. For del Pezzo surfaces, we give a full description of the Terracini loci. Moreover, we present an algorithm to explicitly compute the Terracini loci of a given variety. 
\end{abstract}

\section{Introduction}
Secant varieties are a fruitful meeting ground between classical, projective geometry and modern applied mathematics. They provide solid geometric foundations for concepts like rank, decompositions and identifiability with respect to a variety, and find countless applications in mathematics as well as in physics, chemistry, statistics and computer science. A textbook reference on the connections between secant varieties and their applications is \cite{landsberg}.

Recently the tree of secant varieties sprouted a new branch: the \emph{Terracini loci} of a variety. Introduced in \cite{BC} and \cite{BBS}, they can be considered variations on the theme of secant varieties. Roughly speaking, given a nondegenerate projective variety $X$, a set of $r$ points of $X$ belongs to the $r$-Terracini locus of $X$ if the $r$ tangent spaces at these points are linearly dependent, that is if they span a linear space of dimension smaller than expected. Far from a purely abstract construction, the study of Terracini loci is motivated by very concrete questions about well-posedness of algorithms in \cite[Section 3]{BV18}.

Despite many recent results, the Terracini loci of a variety $X$ are still somewhat mysterious, even when $X$ is a smooth curve or surface. In particular, it is an open problem even to determine whether the Terracini loci of a projective variety  are empty or not, with the notable exception of \cite{LM23}, which solves the problem for linearly normal toric varieties.  As soon as the Terracini locus is nonempty, the next pressing question is to characterize it geometrically. Our aim is to provide an answer for some significant instances,
developing tools that shall permit a similar analysis for larger classes of  varieties. After providing the rigorous setup in \Cref{section: preliminaries}, in \Cref{section: basic} we prove basic properties of Terracini loci, such as its behaviour under inclusion and projection.  In our analysis we use tools from projective geometry, as well as standard techniques employed to handle zero-dimensional schemes. 
The paper contains the following contributions:
\begin{itemize}
\item In \Cref{section: curves} we consider smooth curves. \Cref{coroll: curves embedded with non-complete linear systems} gives a sufficient condition on degree and genus in order for the Terracini locus to be empty.
\item We give a detailed description of all Terracini loci of del Pezzo surfaces in \Cref{section: del Pezzo}. 
\item In \Cref{section: veronese surfaces} we study the smallest nonempty Terracini loci of Veronese varieties of any dimension and degree. We give an explicit geometric description in \Cref{propo: all we know about the Veronese}.
\item In \Cref{sec: SegreVeronese} we perform a similar analysis for the smallest nonempty Terracini loci of Segre-Veronese varieties in any number of factors, and describe it in \Cref{theorem: first nonempty SV}.
\item In \Cref{section: algorithms} we consider an algorithmic approach. We provide code to compute a Terracini locus of a variety, given either its ideal or a parametrization.
\end{itemize}

\subsection*{Historical remark} It is interesting to notice that the Terracini locus of a smooth curve $C\subset\p^3$ already appears in the literature. In \cite[Section 5]{Johnsen89} it is called the \emph{scheme of stationary bisecants}. In \cite[Section 3]{piene91} the author considers the intersection of two tangent lines in order to study what sort of projections to $\p^2$ are possible for $C$. In \cite[Section 2]{literature}, the existence of stationary bisecants is related to \emph{focal surfaces}.  \cite[Theorem 28]{AGI} links the Terracini locus of an elliptic normal curve and the rank with respect to that curve. We find a more applied viewpoint in \cite[Section 2]{ranstu}, where the union of stationary bisecants is called the \emph{edge surface} and it is employed to compute the convex hull of the real part of $C$. Finally we would like to mention \cite[Section 3]{halfdef}, where the study of tangent spaces with nonempty intersection leads to the definition of \emph{sub-defective} varieties.

\section{Definitions and preliminaries}\label{section: preliminaries}
As it is customary in the study of secant varieties, we work over an algebraically closed field $\kk$ of characteristic 0. Let $X\subset \mathbb{P}^N$ be an irreducible projective variety. It is not restrictive to assume that $X$ is \emph{nondegenerate}, meaning that its linear span is the whole $\p^N$. Set
\[\sigma_r^\circ(X)=\bigcup_{p_1,\ldots,p_r\in X} \langle p_1,\ldots, p_r\rangle\subseteq\p^N.  \]
The \emph{$r$th secant variety} of $X$, denoted by 
$\sigma_r(X)$, is the Zariski closure of $\sigma_r^\circ(X)$ in $\p^N$. By construction, there is a dominant map
\[
X^{r}\times\PP^{r-1}\dashrightarrow\sigma_r(X),\]
sending $((p_1,\dots,p_r),z)$ to the point corresponding to $z\in\langle  p_1,\ldots,p_r\rangle\subset\sigma_r(X)$. Hence
\begin{equation}\label{eq: expected dimension}
 \dim \sigma_r(X)\leq \min\{ r(\dim X+1)-1, N\}.   
\end{equation}
The number on the right hand side of \eqref{eq: expected dimension} is called the \emph{expected dimension} of $\sigma_r(X)$. If the inequality is strict, then we say that $X$ is \emph{$r$-defective}. A classical result to compute the dimension of secant varieties is Terracini's lemma, proven in \cite{terracini} - see \cite[Section 2]{dale} for a modern reference.

\begin{theorem}[Terracini's lemma]\label{thm: Terracini Lemma}
Let $X\subset\p^{N}$ be a nondegenerate irreducible variety. 
If $p_1,\ldots, p_r\in X$ are in general position and $z\in\langle p_1,\ldots, p_{r}\rangle$	is a general point, then the embedded tangent space to $\sigma_r(X)$ at $z$ is
$$ T_z\sigma_r(X) = \langle  T_{p_1}X,\ldots,  T_{p_{r}}X\rangle\subseteq\p^N.$$\end{theorem}
When $p_i$ is smooth we have $\dim T_{p_i}X=\dim X$, so we expect that $\dim \langle  T_{p_1}X,\ldots,  T_{p_{r}}X\rangle=r(\dim X+1)-1$ if $r$ is sufficiently small. This is why we call this number the expected dimension of $\sigma_r(X)$. Terracini's lemma tells us that a variety is $r$-defective if and only if the tangent spaces to $X$ at $r$ general points span a linear space smaller than expected.

 \Cref{thm: Terracini Lemma} makes us wonder what happens if we relax the hypothesis that $p_1,\ldots,p_r$ are in general position and instead we only assume that they are smooth points of $X$. In other words, we are interested in the sets of $r$ smooth points of $X$ such that the span of the tangent spaces is smaller than expected. To be precise, let us fix some notation. Given a variety $X$, we denote by $X_{sm}$ the set of smooth points of $X$. The variety parametrizing unordered sets of $r$ points of $X$ is called the \emph{$r$th symmetric power} of $X$. It is the quotient of $\{(p_1,\dots,p_r)\in X^r\mid p_1,\dots,p_r\mbox{ are distinct}\}$ modulo the symmetric group and it is denoted by $X^{(r)}$. 
\begin{definition}\label{def: Terracini locus with tangent spaces}
Let $X\subset\p^N$ be a nondegenerate irreducible variety. The \emph{$r$-Terracini locus} of $X$ is
\[\Ter_r(X)=\overline{
\{
\{p_1,\ldots,p_r\}\in X_{sm}^{(r)}\mid  \dim\langle  T_{p_1}X,\ldots,  T_{p_{r}}X\rangle<r(\dim X+1)-1
\}}.
\]
The closure is taken in the symmetric power $X_{sm}^{(r)}$.
\end{definition}

\begin{remark}\label{rmk: we assume subabundant and nondefective}
In the study of Terracini loci it is not necessary to consider all values of $r$. For instance $\Ter_1(X)=\varnothing$, as observed also in \cite[Section 2]{BC}. Moreover, when $N+1<r(\dim X+1)$, the condition defining $\Ter_r(X)$ becomes empty 
and so $\Ter_r(X)=X^{(r)}$. For this reason we are only interested in values of $r$ satisfying
\begin{equation}\label{eq: subabundant}
2\le r\le \frac{N+1}{\dim X+1}.
\end{equation}
With this choice of $r$, \Cref{thm: Terracini Lemma} tells us that $X$ is $r$-defective if and only if all sets of $r$ tangent spaces to $X$ are linearly dependent, that is $\Ter_r(X)=X^{(r)}$.
\end{remark}
There is another definition of $\Ter_r(X)$, in terms of the cohomology of schemes of double points. If $S=\{p_1,\dots,p_r\}\subset X$ is a set of smooth points with ideal sheaf $I_{S,X}$, we denote by $2S$ the \emph{scheme of double points} supported at $S$. In other words, $2S$ is the subscheme of $X$ defined by the ideal sheaf $I_{2S,X}=(I_{S,X})^2$. Observe that
\begin{align*}
\codim&\langle  T_{p_1}X,\ldots,  T_{p_{r}}X\rangle =\dim\{H\in H^0(\oo_{\p^N}(1))\mid H\supseteq T_{p_i}X\mbox{ for every } i\in\{1,\ldots,r\}\}\\
&=\dim\{H\cap X\in H^0(\oo_{X}(1))\mid H\supseteq T_{p_i}X\mbox{ for every } i\in\{1,\ldots,r\}\}\\
&=\dim\{D\in H^0(\oo_{X}(1))\mid D\mbox{ is singular at }p_1,\ldots,p_r\}\\
&=h^0(I_{2S, X}\ot\oo_X(1)) 
= h^0(I_{2S, X}(1)).
\end{align*}
This means that \Cref{def: Terracini locus with tangent spaces} is equivalent to the following one.
\begin{definition}\label{def: Terracini locus with cohomology}
Let $X\subset\p^N$ be a nondegenerate irreducible variety. The \emph{$r$-Terracini locus} of $X$ is
\begin{align*}
\Ter_r(X)%&=\overline{\{S\in X_{sm}^{(r)}\mid  h^1 (I_{2S,X}(1))>0\}}\nonumber\\&
=
\overline{
\{
S\in X_{sm}^{(r)}\mid  h^0 (I_{2S,X}(1))>N+1-r(\dim X+1)
\}}.
\end{align*}
The closure is taken in the symmetric power $X_{sm}^{(r)}$. Notice that when $h^1(\oo_{X}(1))=0$, then the structure exact sequence of sheaves
\[0\to I_{2S,X}(1)\to\oo_X(1)\to\oo_{2S}\to 0\]
implies that $\Ter_r(X)=\overline{
\{
S\in X_{sm}^{(r)}\mid  h^1 (I_{2S,X}(1))>0
\}}$. This will be the case in Sections \ref{section: del Pezzo}, \ref{section: veronese surfaces} and \ref{sec: SegreVeronese}.
\end{definition}

In the literature we find a few slightly different definitions. In \cite[Definition 2.1]{CG23} the set $S$ is required to be linearly independent. Some require that $I_{2S,X}(1)$ has not only a positive $h^1$, but also a positive $h^0$, like \cite[Definition 1.7]{BBS}, while other focus on a notion of minimality for the set of points, namely that $h^1(I_{2S',X}(1))=0$ for all $S'\subset S$, see \cite[Definition 3.1]{BBS} and \cite{BB23}. Here we follow the definition of \emph{closed Terracini locus} introduced in \cite[Definition 2.1]{BC}, which is more general because it allows us to consider not only $\oo_X(1)$ but any line bundle on $X$. However, in this paper we are interested in the very ample line bundle embedding $X$ in $\p^N$.

\begin{remark}\label{remark: bitangents}Under numerical assumption \eqref{eq: subabundant} we can describe the Terracini locus in terms of bitangent, tritangent and more generally multi-tangent linear spaces. Indeed,
\begin{align*}
\dim&\langle  T_{p_1}X,\ldots,  T_{p_{r}}X\rangle<r(\dim X+1)-1\\
&
\Leftrightarrow \exists\ L=\p^{r(\dim X+1)-2}\subseteq \p^N \mbox{ such that } L\supseteq T_{p_1}X,\ldots,T_{p_r}X\\
& \Leftrightarrow\exists\ L=\p^{r(\dim X+1)-2}\subseteq\p^N \mbox{ tangent to } X \mbox{ at } p_1,\ldots,p_r,
\end{align*}
hence we can interpret the Terracini locus as
\begin{equation*}
\Ter_r(X)=\overline{\{\{p_1,\ldots,p_r\}\in X_{sm}^{(r)}\mid \exists\  \p^{r(\dim X+1)-2}\subseteq\p^N \mbox{ tangent to } X \mbox{ at }p_1,\ldots,p_r\}}.
\end{equation*}
\end{remark}

We conclude this section by recalling another interpretation of the Terracini locus. Let
\[A\sigma_r^\circ=\{(q,S)\in \mathbb{P}^N\times X_{sm}^{(r)} \mid q\in \langle S \rangle =\p^{r-1} \},
\]
its closure, denoted $A\sigma_r(X)$, is the \emph{abstract $r$th secant variety} of $X$. Let $\pi_1 : A\sigma_r(X)\rightarrow \sigma_r(X)$ be the first projection map. As underlined in \cite[Section 1]{BC}, understanding the $r$-Terracini locus of $X$ corresponds to understanding where the differential of $\pi_1$ drops rank in $A\sigma_r^\circ$.

\section{Basic geometric properties}\label{section: basic}
In this section we prove some simple but fundamental properties of $\Ter_r(X)$. We will see what happens to the Terracini loci of a variety under basic geometric operations such as intersection, inclusion or projection from a point. We start by showing that $\dim\Ter_r(X)$ is increasing with respect to $r$.

\begin{proposition}\label{propos: Ter are nested}
Let $X\subset\p^N$ be a nondegenerate irreducible variety and let $r$ be an integer satisfying \eqref{eq: subabundant}. If $\Ter_r(X)\neq \varnothing$, then $ \dim\Ter_{r+1}(X)\geq \dim\Ter_r(X)+\dim X $.

\begin{proof}
Assume that $S\in\Ter_r(X)$ and let $x\in X_{sm}\setminus S$. Since the scheme $2S$ fails to impose independent conditions on $\oo_X(1)$, a fortiori the scheme $2S+2x$ fails as well, so  $S\cup\{x\}\in\Ter_{r+1}(X)$. Therefore we can define a rational map
\begin{center}\begin{tabular}{ccc}
$\Ter_r(X)\times X_{sm}$&$\dashrightarrow $&$\Ter_{r+1}(X)$\\
$(S,x)$&$\mapsto$ & $S\cup \{x\}$.
\end{tabular}\end{center}
It is defined outside of the locus $\{(S,x)\mid x\in S\}$. In order to conclude, it suffices to prove that the fibers are zero-dimensional. Assume that there are two sets $S,S'\in\Ter_r(X)$ and two points $x\in X_{sm}\setminus S$ and $y\in X_{sm}\setminus S'$ such that $S\cup\{x\}=S'\cup\{y\}$. Then $(S,x)$ and $(S',y)$ only differ by a permutation of the elements of $S\cup\{x\}$. In particular, there are only finitely many preimages.
\end{proof}
\end{proposition}

Next we prove that the operation of taking the Terracini locus is compatible with inclusions.

\begin{proposition}\label{propos: Ter of a subvariety}
Let $X\subset\p^N$ be a nondegenerate irreducible variety and let $Y$ be an irreducible subvariety of $X$. We consider the Terracini locus of $Y$ included in its linear span $\langle Y\rangle$. Let $r$ be an integer such that 
\[
2\le r\le \frac{N+1}{\dim X+1}.
\]
Then $\Ter_r(Y)\subseteq\Ter_r(X)$.
\begin{proof} Notice that, even under our numerical assumption, it may happen that $r>\frac{\dim\langle Y\rangle+1}{\dim Y+1}$. In that case, following the definition we have $\Ter_r(Y)=Y^{(r)}$. Let $\{p_1,\ldots,p_r\}\in\Ter_r(Y)$. By definition
\[\dim\langle  T_{p_1}Y,\ldots,  T_{p_{r}}Y\rangle<r(\dim Y+1)-1.\]
Without loss of generality we assume that $T_{p_1}Y$ intersects $\langle T_{p_2}Y,\ldots, T_{p_r}Y\rangle$. Now recall that $Y_{sm}\subseteq X_{sm}$ and $ T_{p_i}X\supseteq T_{p_i}Y$, hence
$ T_{p_1}X$ intersects $\langle T_{p_2}X,\ldots, T_{p_r}X\rangle$, 
and so $\{p_1,\ldots,p_r\}\in\Ter_r(X)$.
\end{proof}
\end{proposition}

\begin{remark}
As a consequence of \Cref{propos: Ter of a subvariety}, if $Y$ is an irreducible component of $X_1\cap X_2$, then $\Ter_r(Y)\subseteq\Ter_r(X_1)\cap\Ter_r(X_2)$ for all relevant values of $r$.
\end{remark}

Now we wonder how the Terracini locus behaves under a morphism. It is easy to see that linear automorphisms of $\p^N$ preserve the Terracini locus. More precisely, if $X\subset\p^N$ is a nondegenerate, irreducible variety and $f\in\PGL(\p^N)$, then $f$ induces an isomorphism $X^{(r)}\to f(X)^{(r)}$ mapping $\Ter_r(X)$ to $\Ter_r(f(X))$. The next result deals with linear projections. First we need an easy preliminary lemma, which is basically an exercise on Grassmann's formula.

\begin{lemma}\label{general if no 3 on a line}
Let $L_1,\ldots,L_r\subset\p^N$ be linear spaces of dimensions $n_1,\ldots,n_r$ respectively. Then $\dim\langle L_1,\ldots,L_r\rangle<n_1+\cdots+n_r+r-1$ if and only if there exists a linear space $V\subset\p^N$ of dimension $r-2$ such that $V\cap L_i\neq\varnothing$ for every $i\in\{1,\ldots,r\}$.
\begin{proof}
First assume that such $V$ exists. For every $i\in\{1,\ldots,r\}$, pick a point $v_i\in V\cap L_i$ and let $x_{i,1},\ldots,x_{i,n_i}\in L_i$ such that $L_i=\langle v_i,x_{i,1},\ldots,x_{i,n_i}\rangle$. Then
\[\dim\langle L_1,\ldots,L_r\rangle\le \dim \langle V,x_{1,1},\ldots,x_{1,n_1},\ldots,x_{r,1},\ldots,x_{r,n_r}\rangle\le \dim V+n_1+\cdots+n_r.\]
On the other hand, let us assume that $\dim\langle L_1,\ldots,L_r\rangle<n_1+\cdots+n_r+r-1$ and let us show the existence of $V$ by induction on $r$. If $r=2$, the Grassmann formula in the projective space tells us that $$\dim(L_1\cap L_2)=n_1+n_2-\dim\langle L_1,L_2\rangle\ge 0,$$
so it is enough to take $V$ to be any point in $L_1\cap L_2$. Now assume that $r\ge 3$ and that $\dim\langle L_1,\ldots,L_r\rangle\le n_1+\cdots+n_r+r-2$. Let $L=\langle L_1,\ldots,L_{r-1}\rangle$. There are two possibilities. If $\dim L< n_1+\cdots+n_{r-1}+r-2$, then by induction hypothesis there exists $W=\p^{r-3}$ that intersects $L_i$ for every $i\in\{1,\ldots,r-1\}$. Now it is enough to pick a point $x\in L_r\setminus L$ and take $V=\langle x,W\rangle$. If $\dim L= n_1+\cdots+n_{r-1}+r-2$, then once again we apply Grassmann's formula to get
\begin{align*}
\dim(L\cap L_r)
=\dim L+\dim L_r-\dim\langle L,L_r\rangle \ge 0,
\end{align*}
hence there is a point $x\in L_r$ and $r-1$ points $x_1\in L_1,\ldots,x_{r-1}\in L_{r-1}$ such that their span $V=\langle x_1,\ldots,x_{r-1}\rangle$ contains $x$. Then  $V=\p^{r-2}$ intersects $L_1,\ldots,L_r$.
\end{proof}
\end{lemma}

Before we state our result on linear projections, we recall that the \emph{tangential variety} of a variety $X\subset\p^N$ is
\[
\tau(X)=\overline{\bigcup_{x\in X_{sm}}T_xX}\subseteq\p^N.
\]
\begin{proposition}\label{propos: projections respect the terracini locus}
Let $N\ge 3$. Let $p\in\p^N$ and call $\pi:\p^N\setminus\{p\}\to\p^{N-1}$ the projection from $p$. Let $X\subset\p^N$ be a nondegenerate irreducible variety and let $r\in\N$ satisfying \[2\le r\le \frac{N}{\dim X+1}.\]
Let $S\subset X$ be a set of $r$ points.
\begin{enumerate}[label=(\arabic*),ref=(\arabic*)]
\item\label{item 1 proj} If $p\notin\sigma_2(X)$, then $S\in\Ter_r(X)\Rightarrow\pi(S)\in\Ter_r(\pi(X))$.
\item\label{item: center of projection away from tangential} If $p\notin\sigma_r(\tau(X))$, then $S\in\Ter_r(X)\Leftrightarrow\pi(S)\in\Ter_r(\pi(X))$.
\end{enumerate}
\begin{proof} Denote $Y=\pi(X)$. First of all,  we want to make sure that it makes sense to talk about $\Ter_r(Y)$. For both cases of the statement, our choice of $p$ ensures that $\pi$ restricts to a linear isomorphism between $X_{sm}$ and $Y_{sm}$. In particular, $X$ and $Y$ have the same dimension and are both irreducible and nondegenerate. Now, let $S=\{x_1,\ldots,x_r\}$ be a general element of $\Ter_r(X)$. Then $x_1,\ldots,x_r$ are $r$ smooth points of $X$ and $p\notin T_{x_i}X$ for any $i\in\{1,\ldots,r\}$. If we call $y_i=\pi(x_i)$, then $y_1,\ldots,y_r$ are smooth points for $Y$ and $T_{y_i}Y=\pi(T_{x_i}X)$. Since $S%\{x_1,\ldots,x_r\}
\in\Ter_r(X)$ is general,  then $T_{x_1}X$ intersects $\langle T_{x_2}X,\ldots,T_{x_r}X\rangle$. As a consequence, $T_{y_1}Y$ intersects $\langle T_{y_2}Y,\ldots,T_{y_r}Y\rangle$. Hence $\{y_1,\ldots,y_r\}\in\Ter_r(Y)$. This proves statement \ref{item 1 proj} and one implication of statement \ref{item: center of projection away from tangential}.

Consider now the converse statement of \cref{item: center of projection away from tangential}. Let $n=\dim X=\dim Y$. Assume by contradiction that there exist $x_1,\ldots,x_r\in X$ such that $\{x_1,\ldots,x_r\}\notin\Ter_r(X)$ but their images $y_1=\pi(x_1),\ldots, y_r=\pi(x_r)$ satisfy $\{y_1,\ldots,y_r\}\in\Ter_r(Y)$. This means that $\dim\langle T_{x_1}X,\ldots,T_{x_r}X\rangle=rn+r-1$ but $\dim\langle T_{y_1}Y,\ldots,T_{y_r}Y\rangle<rn+r-1$. By \cref{general if no 3 on a line}, there exist $t_1\in T_{x_1}X,\ldots, t_r\in T_{x_r}X$ such that their span $\langle t_1,\ldots,t_r\rangle$ has dimension $r-1$, but the span of their images $V=\langle \pi(t_1),\ldots,\pi(t_r)\rangle$ has dimension $r-2$. By definition of $\pi$, we have that $p\in \pi^{-1}(V)=\langle t_1,\ldots,t_r\rangle$. Since $t_1,\ldots,t_r\in\tau(X)$, this implies that $p\in\sigma_r(\tau(X))$, a contradiction.
\end{proof}
\end{proposition}

\section{Terracini loci of curves}\label{section: curves}

This section is devoted to the study of Terracini loci of curves. We address a most basic problem: determining whether the Terracini loci are empty or not. Even for curves, it is not easy to get a complete answer. We give a sufficient condition for emptiness in \Cref{coroll: curves embedded with non-complete linear systems}. When the ambient space has odd dimension, \cite[Proposition 7]{BC23} classifies curves for which every Terracini locus is empty. In \Cref{propos: odd-dimensional ambient space} we present a slight improvement of this result, while in \Cref{corol: rational and elliptic normal curves} we show that this classification does not hold when the ambient space has even dimension.

\begin{proposition}\label{pro: empty terracini for curves embedded with complete linear systems}
Let $C$ be a smooth irreducible curve of genus $g$ and let $D$ be a very ample divisor of $C$.
Consider the embedding of $C$ in $\p^N=\p(H^0\oo_C(D)^\vee)$ and let $r\in\{2,\ldots,\rd{\frac{N+1}{2}}\}$. If $2r<N-g+2$, then $\Ter_r(C)=\varnothing$.
\begin{proof}
Recall that in the embedding $C\subset\p(H^0(\oo_C(D))^\vee)$, the hyperplane section of $C$ is equivalent to $D$, that is, $\oo_C(1)=\oo_C(D)$. Let us start by proving that $h^1(\oo_C(D))=0$. The short exact sequence of sheaves
\[0\to I_{C,\p^N}(1)\to\oo_{\p^N}(1)\to\oo_C(1)\to 0\]
gives an exact sequence of vector spaces
\[H^1(\oo_{\p^N}(1))\to H^1(\oo_C(1))\to H^2(I_{C,\p^N}(1)).\]
The left hand side is 0 by \cite[Theorem III.5.1]{hartshorne}, while the right hand side is 0 because $\dim C=1<2$ - see \cite[Theorem III.2.7]{hartshorne}.  Hence $h^1(\oo_C(D))=h^1(\oo_C(1))=0$. Apply Riemann-Roch theorem to the divisor $D$ to get $N+1=h^0(\oo_C(D))=h^1(\oo_C(D))+\deg(D)-g+1$, thus
\begin{equation}\label{eq: d=N+g}
\deg(D)=N+g.
\end{equation}
If $S\subset C$ is a set of $r$ distinct points and $K_C$ is the canonical divisor of $C$, then by definition
\begin{align*}
S\in\Ter_r(C)&\Leftrightarrow h^1(\oo_C(D)\ot I_{2S,C})>0\\
&\Leftrightarrow h^1(\oo_C(D-2S))>0\Leftrightarrow h^0(\oo_C(K_C-D+2S))>0,
\end{align*}
where the last equivalence comes from Serre duality - see \cite[Corollary III.7.8]{hartshorne}. In order to conclude that $\Ter_r(C)$ is empty, it is enough to show that $h^0(\oo_C(K_C-D+2S))=0$. Thanks to \cite[Lemma IV.1.2]{hartshorne}, it suffices to check that $K_C-D+2S$ has negative degree. By combining equality \eqref{eq: d=N+g} and our hypothesis on $r$, we get
\[
\deg(K_C-D+2S)=2g-2-\deg(D)+2r=g-2-N+2r<0.\qedhere
\]
\end{proof}
\end{proposition}

Let us apply \Cref{pro: empty terracini for curves embedded with complete linear systems} to low-genus curves.
\begin{example}
\label{corol: rational and elliptic normal curves}If $C\subset\p^N$ is the rational normal curve of degree $N$, then $\Ter_r(C)=\varnothing$ for every $r\in\{2,\ldots,\rd{\frac{N+1}{2}}\}$. Hence we recover the fact that $r$ double points always impose independent conditions on divisors of $\p^1$. When $N$ is odd, this property actually characterizes rational normal curves, as shown in \Cref{propos: odd-dimensional ambient space}. 

If $C\subset\p^N$ is a degree $N+1$ elliptic normal curve, then $\Ter_r(C)=\varnothing$ for every $r\in\{2,\ldots,\rd{\frac{N}{2}}\}$. When $N$ is even, this means that $\Ter_r(C)=\varnothing$ for all values of $r$ satisfying \eqref{eq: subabundant}, hence \Cref{propos: odd-dimensional ambient space} does not generalize to even-dimensional spaces.  When $N$ is odd, this means that all relevant Terracini loci of $C$ are empty, except for the last one. 
\end{example}

\cref{pro: empty terracini for curves embedded with complete linear systems} works only for curves embedded with a complete linear system. We combine it with \cref{propos: projections respect the terracini locus} to obtain a more general result.

\begin{corollary}\label{coroll: curves embedded with non-complete linear systems} Let $C$ be a smooth irreducible curve of genus $g$. Consider a divisor $D$ of $C$ and a very ample linear system $\LL\subseteq H^0(\oo_C(D))$ embedding $C$ in $\p(\LL^\vee)=\p^n$. Let $r\in\{2,\ldots,\rd{\frac{n+1}{2}}\}$. If
\[2r<h^0(\oo_C(D))-g+1\mbox{ and } 3r<n+2,\]
then $\Ter_r(C)=\varnothing$.
\begin{proof}
Let $N=h^0(\oo_C(D))-1$ and choose a basis $e_0,\ldots,e_N$ of $H^0(\oo_C(D))^\vee$ such that $\langle e_0,\ldots,e_n\rangle=\LL^\vee$. Then the embedding $C\subset\p(\LL^\vee)$ factors as
\[\begin{tikzcd}
\p(\langle e_0,\ldots,e_N\rangle) \arrow{r}{\pi_N}  & \p(\langle e_0,\ldots,e_{N-1}\rangle) \arrow{r}{\pi_{N-1}}& \ldots \arrow{r}{\pi_{n+1}}& \p(\langle e_0,\ldots,e_{n}\rangle) \arrow[equal]{d}\\
C \arrow[hook]{u}\arrow[hook]{rrr} &  & & \p(\LL^\vee)
\end{tikzcd}
\]
where $\pi_i$ is the projection on the first $i$ coordinates. Call $C_i$ the above embedding of $C$ in $\p(\langle e_0,\ldots,e_i\rangle)$. Thanks to \cref{pro: empty terracini for curves embedded with complete linear systems}, our first hypothesis on $r$ guarantees that $\Ter_r(C_N)=\varnothing$. In order to conclude, we will apply \cref{propos: projections respect the terracini locus}. Each time we project from $\p(\langle e_0,\ldots,e_i\rangle)$ to $\p(\langle e_0,\ldots,e_{i-1}\rangle)$, the $r$-Terracini locus is preserved provided that the center of projection is outside $\sigma_r(\tau(C_i))$. Since $C$ is a curve, $\tau(C_i)$ is a surface so $\dim\sigma_r(\tau(C_i))\le 3r-1$. Our second hypothesis on $r$ is enough to be sure that $\sigma_r(\tau(C_i))\subsetneq\p(\langle e_0,\ldots,e_i\rangle)$ for every $i\in\{n+1,\ldots,N\}$. This means that we can preserve the emptiness of $\Ter_r(C_i)$ along all the projections, and we conclude that $\Ter_r(C_n)=\varnothing$.
\end{proof}
\end{corollary}
The bound from \cref{coroll: curves embedded with non-complete linear systems} is not sharp. Indeed, let $C\subset \p^7$ be a smooth, rational curve of degree 8. The relevant values of $r$ are 2, 3 and 4. By \cite[Proposition 7]{BC23} we know that $\Ter_4(C)\neq\varnothing$. Explicit software computations -- see \cref{section: algorithms} -- show that $\Ter_2(C)=\Ter_3(C)=\varnothing$. However, \cref{coroll: curves embedded with non-complete linear systems} only goes as far as telling us that $\Ter_2(C)=\varnothing$.

In the case of an odd dimensional ambient space we can go further. The following characterization extends \cite[Proposition 7]{BC23} and vastly generalizes the result on curves in $\p^3$ from \cite[Section 5]{Johnsen89}.
\begin{proposition}\label{propos: odd-dimensional ambient space}
    Let $C\subset\p^{2k+1}$ be a smooth, nondegenerate, irreducible curve. Then
    $$ C \hbox{ is a rational normal curve}\Leftrightarrow \Ter_{k+1}(C)=\varnothing \Leftrightarrow\Ter_{k+1}(C) \hbox{ is a finite set}.   $$
\end{proposition}
\begin{proof}
The first equivalence of the statement is \cite[Proposition 7]{BC23}. Now we want to prove that if $\Ter_{k+1}(C)$ is finite, then it is actually empty. Let us distinguish two cases. If $\Ter_k(C)\neq \varnothing$, then  by Proposition \ref{propos: Ter are nested} $\dim \Ter_{k+1}(C)\geq \dim \Ter_k(C)+\dim C\geq 1$, and this is a contradiction. Hence now we assume that $\Ter_{k}(C)$ is empty. 
    Let $S=\{x_1,\ldots,x_k\}\subset C$ be a set of $k$ points of $C$ and let 
\[L_S=\langle T_{x_1}C,\ldots,T_{x_k}C\rangle.
\]
%Since $\Ter_{k+1}(C)$ is empty, by Proposition \ref{propos: Ter are nested} $\Ter_{k}(C)$ is empty as well, 
Since $\Ter_k(C)=\varnothing$,  $S\notin\Ter_{k}(C)$ and therefore $\dim L_S=2k-1$. Let
\[\h_S=\p(H^0(I_{L_S,\p^{2k+1}}(1)))\]
be the pencil of hyperplanes containing $L_S$. 
We define $\varphi_S \colon C\dashrightarrow \h_S$ by
\[x\mapsto \langle x, L_S\rangle.\]
This construction is considered from a different perspective also in the proof of \cite[Proposition 7]{BC23}. The authors proved that the map $\varphi_S$ is defined on $C\setminus S$ and this allows to extend $\varphi_S$ to a morphism $\phi_S:C\to\h_S$. With the given construction, assume that $C$ is not the rational normal curve. For a generic choice of $k$ points $S=\{x_1,\ldots,x_k\}\subset C$, there exists at least one point $x\in C$ such that $\{x,x_1,\ldots,x_k\}\in \Ter_{k+1}(C)$, which implies $\dim \Ter_{k+1}(C)\geq k$.
\end{proof}

In contrast to \Cref{propos: odd-dimensional ambient space}, \Cref{corol: rational and elliptic normal curves} shows that, when the dimension of the ambient space is even, there exist non-rational curves such that all the Terracini loci are empty. We conclude this section with an example of smooth rational curve with the same property.

\begin{example}\label{example: rational quintic in P4 without Terracini points}
Let $C\subset\p^4$ be the image of the map $f:\p^1\to\p^4$ defined by
$$[x:y]\mapsto [x^5:x^4y:x^3y^2:xy^4:y^5].$$
It is a linear projection of the rational normal quintic in $\p^5$. It is smooth, rational and nondegenerate, and the only relevant value of $r$ is 2. A software computation -- see Section \ref{section: algorithms} -- shows that $\Ter_2(C)=\varnothing$.
\end{example}

\section{Terracini loci of del Pezzo surfaces}\label{section: del Pezzo}
In this section we study Terracini loci of a classical family of varieties. Del Pezzo surfaces are smooth surfaces whose anticanonical divisor is very ample, and have been studied since the nineteenth century. Since our setting requires to fix an embedding, we will study del Pezzo surfaces embedded by the anticanonical linear system. By \cite[Proposition IV.9 and Exercise V.21(2)]{beauville}, we are dealing with a finite list of examples: a del Pezzo surface is either the blow-up of $\p^2$ at at most $6$ points, or the Segre-Veronese surface $\PP^1\times \PP^1$ embedded by $\oo(2,2)$ in $\p^8$. By \cite[Corollary 2.3]{CGG05}, the latter is 3-defective, while the description of its 2-Terracini locus is a special case of \Cref{theorem: first nonempty SV}. For this reason, in this section we will only consider blow-ups
$$X=\Bl_Z\p^2\xrightarrow{\mu}\p^2$$
of the plane, where $Z=\{z_1,\dots,z_t\}\subset\PP^2$ is 
 a set of general points. Recall that the divisor group of $X$ has $t+1$ generators: the pullback $\mu^*(\ell)$ of the class of a line $\ell\subset\p^2$ and the exceptional divisors $E_1,\dots,E_t$. According to \cite[Remark IV.10(1)]{beauville}, the anticanonical linear system on $X$ is the pullback of the system of cubics containing $Z$, so it embeds $X$ in $\p^{9-t}$. We denote by $\LL=3\mu^*(\ell)-E_1-\cdots-E_t$ the very ample linear system embedding $X$. Following numerical condition \eqref{eq: subabundant}, we are interested in $2\le r\le \frac{10-t}{3}$, and so $t\le 4$. 
For $t=1$ we will study $\Ter_2(X)$ and $\Ter_3(X)$, while for $t\in\{2,3,4\}$ we will only consider $\Ter_2(X)$. 

For our convenience, we distinguish two kinds of elements of the Terracini locus of a del Pezzo surface: sets that do not intersect any exceptional divisor, and sets that do. In the first case we can reduce the problem to an exercise about plane cubics. 
\begin{lemma}\label{lemma: speciality Ter2 di delPezzo}
Let $Z=\{z_1,\ldots,z_t\}\subset\p^2$ be a set of $t$ general points and let  $\{q_1,q_2\}\subset\p^2\setminus Z$ be a set of two  points.
\begin{enumerate}
\item If $t\in\{1,2,3\}$, then $h^1(I_{Z+2q_1+2q_2}(3))>0$ if and only if the line $\langle q_1,q_2\rangle$ contains at least a point of $Z$.
\item If $t=4$, then $h^1(I_{Z+2q_1+2q_2}(3))>0$ if and only if either $\langle q_1,q_2\rangle$ contains a point of $Z$ or $Z\cup S$ is contained on a conic.
\end{enumerate}
\begin{proof}
The expected dimension of $H^0(I_{Z+2q_1+2q_2}(3))$ is $4-t$. By B\'ezout's theorem, the line $\langle q_1,q_2\rangle$ is a base component of $I_{Z+2q_1+2q_2}(3)$. If we call $Z'=Z\setminus \langle q_1,q_2\rangle$, then
\[h^0(I_{Z+2q_1+2q_2}(3))=h^0(I_{Z'+q_1+q_2}(2)).\]
If $\langle q_1,q_2\rangle$ contains at least a point of $Z$, then $Z'$ has at most $t-1$ points and therefore $h^0(I_{Z'+q_1+q_2}(2))$ $\ge 5-t$, greater than expected. In a similar way, if $t=4$ and $Z\cup \{q_1,q_2\}$ is contained in a conic, then $h^0(I_{Z'+q_1+q_2}(2))$ is greater than expected. Conversely, assume that $Z=Z'$. If $t\le 3$, then $Z\cup \{q_1,q_2\}$ is a set of at most five linearly general points, so they impose independent conditions to conics and $h^0(I_{Z+2q_1+2q_2}(3))=h^0(I_{Z'+q_1+q_2}(2))=4-t$. If $t=4$ and $Z\cup \{q_1,q_2\}$ is not contained in a conic, then $h^0(I_{Z+2q_1+2q_2}(3))=h^0(I_{Z'+q_1+q_2}(2))=0$, as expected.
\end{proof}
\end{lemma}

\begin{lemma}\label{lemma:speciality for ter3 dP1}
Let $z\in\p^2$ and let $Q=\{q_1,q_2,q_3\}\subset\p^2\setminus\{z\}$ be a set of three points. Then $h^0(I_{z+2Q}(3))\neq 0$ if and only if $q_1,q_2,q_3$ are collinear or $z\in \langle q_1,q_2\rangle\cup \langle q_1,q_3\rangle\cup \langle q_2,q_3\rangle$.
\begin{proof}
By Bézout's theorem, the lines $\langle q_1,q_2\rangle$, $\langle q_1,q_3\rangle$ and $\langle q_2,q_3\rangle$ are base components of the linear system $H^0(I_{z+2Q}(3))$. There are two possibilities: either  $q_1,q_2,q_3$ are collinear or not. If $q_1,q_2,q_3$ lie on the same line $\ell$, then
$$H^0(I_{z+2Q}(3))=\{
2\ell+\ell'
 \mid \ell' \mbox{ is a line containing } z
\}\neq 0.$$
If $q_1,q_2,q_3$ are not collinear, then the lines $\langle q_1,q_2\rangle$, $\langle q_1,q_3\rangle$ and $\langle q_2,q_3\rangle$ are distinct, so their union is the only element of $H^0(I_{2Q}(3))$. Then $h^0(I_{z+2Q}(3))\neq 0$ if and only if $z\in\langle q_1,q_2\rangle\cup\langle q_1,q_3\rangle\cup\langle q_2,q_3\rangle$.
\end{proof}\end{lemma}

Lemmas \ref{lemma: speciality Ter2 di delPezzo} and \ref{lemma:speciality for ter3 dP1} characterize elements $S$ of the 2-Terracini locus and the 3-Terracini locus such that no point of $S$ belongs to any exceptional divisor. Now we study the other case, when at least one point of $S$ lies on one of the exceptional divisors. If $C\subset\p^2$ is a plane curve, we indicate by $\tilde{C}=\overline{\mu^{-1}(C)\setminus (E_1\cup\cdots\cup E_t)}$ its strict transform on $X$.

\begin{remark}\label{remark: cubics singular contains E}
Let $z\in\p^2$ and let $E$ be the exceptional divisor of $\Bl_z\p^2\xrightarrow{\mu}\p^2$. Let $e\in E$. Let $Y\in 
H^0(I_{2e}\ot(\mu^*(3\ell)-E))$. Call $C=\mu_*(Y)$. Then $C$ is a plane cubic and $\tilde{C}\sim \mu^*(3\ell)-mE$, where $m$ is the multiplicity of $C$ at $z$. 
Hence $Y=(m-1)E\cup \tilde{C}$. Since we are assuming that $Y\in 
H^0(\mu^*(3\ell)-E)$, we deduce that $m\ge 2$ and therefore $C$ is singular at $z$. For the general element of this linear system, $m=2$ and $C$ is a nodal cubic passing twice through $z$, and at least one of the tangent directions is the one corresponding to $e$.
\end{remark}

In general, if $p_j\in E_i$ for some $i\in\{1,\dots,t\}$, then we indicate by $L_j\subset\p^2$ the line through $z_i$ corresponding to $p_j$. In the following example we describe four different configurations of points belonging to the $2$-Terracini locus of $X$ with the property that at least one of the points belongs to an exceptional divisor. 
\begin{example}\label{example: S touches E and S in Ter} Let $S=\{p_1,p_2\}\in\Ter_2(X)$ such that $S\cap(E_1\cup\cdots\cup E_t)\neq\varnothing$. Recall that $H^0(I_{2S}\ot \LL)$ has expected dimension $4-t$. 

\begin{enumerate}
\item\label{item: S subset E} Since the exceptional divisors are lines by \cite[Proposition IV.12]{beauville}, \Cref{propos: Ter of a subvariety} implies that $E_i^{(2)}\subseteq\Ter_2(X)$ for every $i\in\{1,\dots,t\}$.
\item\label{item: p1 in E1, p2 in E2 and Lp1 contains two poins of Z} Assume $t\ge 2$. Suppose that $p_1\in E_1$ and $p_2\in E_2$.  
If $z_2\in L_1$, then $E_1+E_2+\tilde{L_1}+\tilde{B}$ is an element of $H^0(I_{2S}\ot \LL)$ whenever $B$ is a conic tangent to $L_2$ at $z_2$ and containing $Z\setminus\{z_2\}$. 
The dimension of the linear system of such conics is at least $6-2-(t-1)>4-t$, so $S\in\Ter_2(X)$. 
\item\label{item: p1 in E1, p2 notin E2 p2 in L1} Suppose that $p_1\in E_1$ and $p_2\notin E_1\cup\cdots\cup E_t$.
If $\mu(p_2)\in L_1$, then $E_1+\tilde{L_1}+\tilde{B}$ is an element of $H^0(I_{2S}\ot \LL)$ whenever $B$ is a conic containing $Z\cup\{\mu(p_2)\}$. The dimension of the linear system of such conics is at least $6-1-t>4-t$, so $S\in\Ter_2(X)$.

\item\label{item: p1 in E1, p2 notin E2 p2 in z1z2} Assume $t\ge 2$. Suppose that $p_1\in E_1$ and $p_2\notin E_1\cup\cdots\cup E_t$ and $\mu(p_2)\in\langle z_1,z_2\rangle$.  If we call $L=\langle z_1,z_2\rangle$, then $E_1+\tilde{L}+\tilde{B}$ is an element of $H^0(I_{2S}\ot \LL)$ whenever $B$ is a conic tangent to $L_1$ at $z_1$, passing through $\mu(p_2)$ and containing $Z\setminus \{z_1,z_2\}$.  The dimension of the linear system of such conics is at least $6-2-1-(t-2)>4-t$, so $S\in\Ter_2(X)$.
\end{enumerate}
\end{example}

Let us show that in \Cref{example: S touches E and S in Ter} we actually listed all instances of elements $S\in\Ter_2(X)$ such that at least a point of $S$ is on an exceptional divisor.

\begin{lemma}\label{lemma: r=2 S touches E}
Let $t\in\{1,\dots,4\}$. Let $Z=\{z_1,\dots,z_t\}\in\p^2$ be a set of general points and let $X=\Bl_Z\p^2\xrightarrow{\mu}\p^2$ be the del Pezzo surface in $\p^{9-t}$. Let $S=\{p_1,p_2\}\subset X$ be a set of two points such that $S\cap (E_1\cup\cdots\cup E_t)\neq \varnothing$. Then $S\in\Ter_2(X)$ if and only if $S$ is one of the sets described by \Cref{example: S touches E and S in Ter}, up to a relabeling of $Z$.
\begin{proof}
Without loss of generality we assume that $p_1\in E_1$. First we suppose that $S\subset E_1\cup\cdots\cup E_t$. If $S\subset E_i$ for some $i\in\{1,\dots,t\}$, we are in the situation of \Cref{example: S touches E and S in Ter}\eqref{item: S subset E} and there is nothing to prove. Assume then that $t\ge 2$ and $p_2\in E_2$. Suppose that we are not in the situation of \Cref{example: S touches E and S in Ter}\eqref{item: p1 in E1, p2 in E2 and Lp1 contains two poins of Z}, so $z_2\notin L_1$ and $z_1\notin L_2$. According to \Cref{remark: cubics singular contains E}, the general element of $H^0(I_{2S}\ot\LL)$ is of the form $E_1+E_2+\tilde{C}$, where $C$ is a plane cubic singular at $z_1$ and $z_2$. If we call $L=\langle z_1,z_2\rangle$, then our assumption is that $L_1\neq L\neq L_2$ and $C=L+B$, where $B$ is a conic tangent to $L_1$ at $z_1$, tangent to $L_2$ at $z_2$ and containing $Z\setminus\{z_1,z_2\}$. Under our assumption these $4+t-2$ linear conditions are independent, so the dimension of the linear system of such conics is $6-(4+t-2)=4-t$ and  $S\notin\Ter_2(X)$.

Now we have exhausted the case $S\subset E_1\cup\cdots\cup E_t$, so we assume that $p_2\notin E_1\cup\cdots\cup E_t$. Suppose that we are not in the situation of \Cref{example: S touches E and S in Ter}\eqref{item: p1 in E1, p2 notin E2 p2 in L1} nor of \Cref{example: S touches E and S in Ter}\eqref{item: p1 in E1, p2 notin E2 p2 in z1z2}, so $\mu(p_2)\notin L_1$ and $\mu(p_2)\notin\langle z_1,z_i\rangle$ for any $i\in\{1,\dots, t\}$. Again by \Cref{remark: cubics singular contains E}, the general element of $H^0(I_{2S}\ot\LL)$ is of the form $E_1+\tilde{C}$, where $C$ is a plane cubic singular at $z_1$ and $\mu(p_2)$. If we call $L=\langle z_1,\mu(p_2)\rangle$, then $C=L+B$, where $B$ is a conic tangent to $L_1$ at $z_1$, passing through $\mu(p_2)$ and containing $Z\setminus\{z_1\}$. Under our assumption these $2+1+t-1$ linear conditions are independent, so the dimension of the linear system of such conics is $6-2-t=4-t$ and  $S\notin\Ter_2(X)$. 
\end{proof}\end{lemma}

Thanks to \Cref{lemma: speciality Ter2 di delPezzo} and \Cref{lemma: r=2 S touches E}, now we know exactly what sets belong to $\Ter_2(X)$. In order to determine the irreducible components of $\Ter_2(X)$, we want to understand which of these sets are specializations of others. 
\begin{definition}\label{def:  Y W V}
Let $t\in\{1,\dots,4\}$. Let $Z=\{z_1,\dots,z_t\}\in\p^2$ be a set of general points and let $X=\Bl_Z\p^2\xrightarrow{\mu}\p^2$ be the del Pezzo surface in $\p^{9-t}$. Recall that, if $p_i\in E_k$, we call $L_i\subset\p^2$ the line through $z_k$ corresponding to $p_i$. Motivated by \Cref{lemma: speciality Ter2 di delPezzo} and \Cref{example: S touches E and S in Ter}\eqref{item: p1 in E1, p2 notin E2 p2 in L1}, for every $i\in\{1,\dots,t\}$ we define
\begin{align*}
Y_i&=\{\{p_1,p_2\}\in X^{(2)}\mid p_1,p_2\notin E_1\cup\cdots\cup E_t\mbox{ and } z_i\in\langle \mu(p_1),\mu(p_2)\rangle\},\\
U&=\{
\{p_1,p_2\}\in X^{(2)}\mid p_1,p_2\notin E_1\cup\cdots\cup E_{t}
\mbox{ and }
h^0(I_{Z\cup\{\mu(p_1),\mu(p_2)\}}(2))>0\}\mbox{ and}\\
V_i&=\{\{p_1,p_2\}\in X^{(2)}\mid p_1\in E_i\mbox{, }p_2\notin E_1\cup\cdots\cup E_t\mbox{ and } \mu(p_2)\in L_1\} .
\end{align*}
In a similar way, motivated by \Cref{example: S touches E and S in Ter}\eqref{item: p1 in E1, p2 in E2 and Lp1 contains two poins of Z} and \Cref{example: S touches E and S in Ter}\eqref{item: p1 in E1, p2 notin E2 p2 in z1z2}, if $t\geq 2$ and $i,j\in\{1,\ldots,t\}$ are distinct indices, then we define
\begin{align*}
W_{ij}&=\{\{p_1,p_2\}\in X^{(2)}\mid p_1\in E_i\mbox{, }p_2\in E_j\mbox{ and } z_j\in L_1\} \mbox{ and}\\
T_{ij} &=\{\{p_1,p_2\}\in X^{(2)}\mid p_1\in E_i\mbox{, }p_2\notin E_1\cup\cdots\cup E_t\mbox{ and } \mu(p_2)\in\langle z_i,z_j\rangle\}.
\end{align*}

Notice that each of these sets is irreducible. Since $Y_i$ is isomorphic to a dense open subset of $\p^2\times \p^1$, it has dimension 3. In a similar way we observe that $\dim(U)=3$.
\end{definition}

\begin{lemma}\label{lemma: Vi is limit of Yi} Consider the notation of \Cref{def:  Y W V} and let $i\in\{1,\dots,t\}$. Then
\begin{enumerate}
\item $V_i\subseteq \overline{Y_i}$.
\item If $t=4$, then $U\not\subseteq \overline{Y_i}$ and $Y_i\not\subseteq \overline{U}$.
\item $E_i^{(2)}\not\subseteq \overline{Y_i}$. Moreover, if $t=4$ then $E_i^{(2)}\not\subseteq\overline{U}$.
\end{enumerate}
Furthermore, suppose $t\geq 2$ and $i,j\in\{1,\ldots,t\}$ are distinct.  Then
\begin{enumerate}[resume]
\item $E_i^{(2)}\subseteq \overline{Y_j}$,
\item

$W_{ij}\subseteq\overline{Y_i}$,
\item $T_{ij}\subseteq\overline{Y_j}$.
\end{enumerate}
\begin{proof}
\begin{enumerate}
\item Let $\{p_1,p_2\}\in V_i$. Now let $q$ be a general point of $\tilde{L_1}$, so that $\mu(q)\in L_1=\langle z_i,\mu(p_2)\rangle$. By construction $\langle\mu(q),\mu(p_2)\rangle\ni z_i$, so $\{q,p_2\}\in Y_i$. Now let $q$ move along $\tilde{L_1}$ and specialize $q$ to $\tilde{L_1}\cap E_i$. In this way $q$ specializes to $p_1$, so $\{p_1,p_2\}$ is a limit of elements of $Y_i$.
\item Since $\dim(U)=\dim(Y_i)$, if one of them was contained in the closure of other one they would be the same.
\item 
Let $\{q_1,q_2\}\in Y_i$ and let $L=\langle\mu(q_1),\mu(q_2)\rangle$. The only way to specialize $q_1$ onto $E_i$ is to specialize it to a point of $E_i\cap \tilde{L}$. But $\tilde{L}$ intersects $E_i$ at one point.   This means that it is impossible to specialize $\{q_1,q_2\}$ to two distinct points of $E_i$ while staying in $\overline{Y_i}$. This means that $E_i^{(2)}\not\subseteq \overline{Y_i}$.  If $t=4$, then something similar happens with $U$. If we start from an element $\{q_1,q_2\}\in U$ and we call $C$ the conic containing $Z\cup\{q_1,q_2\}$, then every specialization of $q_1$ onto $E_i$ will be a point of $E_i\cap \tilde{C}$. Since $C$ is smooth at $z_i$, that intersection consists of only one point, so it is not possible to specialize $\{q_1,q_2\}$ to two distinct points of $E_i$.
\item Let $\{p_1,p_2\}\in E_i^{(2)}$. Let $L$ be a general line through $z_j$ and let $q_1,q_2\in X\setminus(E_1\cup\cdots\cup E_t)$ such that $\mu(q_1)= L\cap L_1$ and $\mu(q_2)=L\cap L_2$. Let $L$ move along the pencil of lines through $z_j$ and specialize $L$ to $\langle z_i,z_j\rangle$. In this way $\mu(q_1)$ approaches $z_i$ along $L_1$ and $\mu(q_2)$ approaches $z_i$ along $L_2$. Hence $q_1$ specializes to $p_1$ and $q_2$ specializes to $p_2$, so $\{p_1,p_2\}$ is a limit of elements of $Y_j$. 
\item Let $\{p_1,p_2\}$ be a general point of $W_{ij}$. Then $p_1\in E_i$, $p_2\in E_j$ and $L_1=\langle z_i,z_j\rangle$. Since it is general we have $L_1\neq L_2$. Let $C\subset\PP^2$ be a smooth conic that is tangent to $L_1$ at $z_i$ and choose a general point $q_1\in\tilde C\setminus(E_1\cup\cdots\cup E_t\cup\tilde{L_2})$. Now let $L=\langle\mu(q_1),z_i\rangle\subset\PP^2$ and $q_2=\mu^{-1}(L\cap L_2)$, and thus $\{q_1,q_2\}\in Y_i$.  Now move $q_1$ along $\tilde C$ toward $p_1$.  Since $L$ approaches $L_1$, $q_2$ moves along $\tilde{L_2}$ toward $p_2$, and so $\{p_1,p_2\}$ is a limit of elements in $Y_i$. 
\item Let $\{p_1,p_2\}\in T_{ij}$, so $z_i\in L_1$ and $\mu(p_2)\in\langle z_i,z_j\rangle$. Choose a general point $q_1\in\tilde{L_1}\setminus E_i$, let $L=\langle\mu(q_1),z_j\rangle$, and pick a general line $L'\subset\PP^2$ through $\mu(p_2)$.  Letting $q_2=\mu^{-1}(L\cap L')$, we have $\{q_1,q_2\}\in Y_j$.  Now move $q_1$ along $\tilde{L_1}$ so that it approaches $p_1$.  Then $L$ will approach $\langle z_i,z_j\rangle$, but since $\mu(p_2)\in\langle z_i,z_j\rangle$, we see that $q_2$ will approach $p_2$ along $\tilde{L'}$, and so $\{p_1,p_2\}$ is a limit of elements in $Y_j$.\qedhere
\end{enumerate}
\end{proof}
\end{lemma}

Summarizing our results, we can describe the $2$-Terracini loci of del Pezzo surfaces.
\begin{proposition}\label{coroll: Ter2 per delpezzos}
Let $t\in\{1,\dots,4\}$. Let $Z=\{z_1,\dots,z_t\}\in\p^2$ be a set of general points and let $X=\Bl_Z\p^2\xrightarrow{\mu}\p^2$ be the del Pezzo surface in $\p^{9-t}$.
\begin{enumerate}
    \item If $t=1$, then $\Ter_2(X)=\overline{Y_1}\cup E_1^{(2)}$ has a component of dimension 3 and a component of dimension 2. 
\item If $t\in\{2,3\}$, then
$\Ter_2(X)=\overline{Y_1}\cup\cdots\cup \overline{Y_t}$
has $t$ components of dimension 3.

\item If $t=4$, then
$\Ter_2(X)=
\overline{Y_1}\cup\cdots\cup \overline{Y_4}\cup\overline U$
has five components of dimension 3.

\end{enumerate}\end{proposition}

We conclude this section by considering the remaining case $t=1$ and $r=3$. In this situation we expect $H^0(I_{2S}\ot (\mu^*(3\ell)-E))=0$.

\begin{lemma}\label{lemma: r=3 S touches E is in ter} Let $z\in\p^2$ and let $X=\Bl_z\p^2\xrightarrow{\mu}\p^2$ be the del Pezzo surface in $\p^8$. Let $E$ be the exceptional divisor. Let $S=\{p_1,p_2,p_3\}\subset X$ be a set of three points such that $S\cap E\neq \varnothing$. Up to permutation of the indices, $S\in\Ter_3(X)$ if and only if either
\begin{enumerate}
    \item $p_1,p_2\in E$, or
\item $p_1\in E$, $p_2,p_3\notin E$ and $L_1\cap \{\mu(p_2),\mu(p_3)\}\neq\varnothing$.
\end{enumerate}
\begin{proof}
Let $S=\{p_1,p_2,p_3\}$. If $E$ contains at least two points of $S$, say $p_1$ and $p_2$, then $\{p_1,p_2\}\in \Ter_2(X)$ by \Cref{example: S touches E and S in Ter}\eqref{item: S subset E}, so a fortiori $S\in\Ter_3(X)$. Now assume that $p_1\in E$ and $p_2,p_3\notin E$. By \Cref{remark: cubics singular contains E}, the general element in $H^0(I_{2S}\ot (\mu^*(3\ell)-E))$ is of the form $E+\tilde{C}$, where $C$ is a plane cubic singular at $\mu(p_2)$, singular at $\mu(p_3)$ and passing twice through $z$, with at least one of the tangent directions corresponding to $p_1$. The lines $\langle\mu(p_2),\mu(p_3)\rangle$, $\langle z,\mu(p_2)\rangle$ and $\langle z,\mu(p_3)\rangle$ are irreducible components of $C$. The only possibility for $S$ to be an element of $\Ter_3(X)$ or, equivalently, for $h^0(I_{2S}\ot (\mu^*(3\ell)-E))$ to be greater than 0, is that one of the three lines is actually $L_1$.
\end{proof}
\end{lemma}

Thanks to \Cref{lemma:speciality for ter3 dP1} and \Cref{lemma: r=3 S touches E is in ter} we know exactly what sets belong to $\Ter_3(X)$. Now we want to understand which of these sets are specialization of others.

\begin{definition}\label{definition: def for ter3 of del pezzo}
Let $z\in\p^2$ and let $X=\Bl_z\p^2\xrightarrow{\mu}\p^2$ be the del Pezzo surface in $\p^8$. Let $E$ be the exceptional divisor. Define
\begin{align*}
Y&=\{\{p_1,p_2, p_3\}\in X^{(3)}\mid p_1,p_2, p_3\notin E\mbox{ and }\mu(p_1),\mu(p_2), \mu(p_3) \mbox{ are collinear}\},\\
T&=\{\{p_1,p_2, p_3\}\in X^{(3)}\mid p_1,p_2\in E\}\mbox{ and}\\
A&=\{\{p_1,p_2, p_3\}\in X^{(3)}\mid p_1\in E\mbox{, $p_2,p_3\notin E$ and } \mu(p_2)\in L_1\}.
\end{align*}
Finally, for $1\le i<j\le 3$ define 
$$B_{ij}=\{\{p_1,p_2, p_3\}\in X^{(3)}\mid p_1,p_2, p_3\notin E\mbox{ and } z\in\langle \mu(p_i),\mu(p_j)\rangle\}.$$
Notice that each of these sets is irreducible.
\end{definition}

\begin{lemma}\label{lemma: specializations for ter3 of del pezzo}
In the notation of \Cref{definition: def for ter3 of del pezzo}, $A\subseteq\overline{B_{12}}$ and $T\subseteq\overline{Y}$.
\begin{proof} Let $\{p_1,p_2,p_3\}\in A$, so that $\mu(p_2)\in L_1$. Let $q$ be a general point of $\tilde{L_1}$, so that $\mu(q)\in L_1=\langle z,\mu(p_2)\rangle$. Then $\{q,p_2,p_3\}\in B_{12}$. Let $q$ move along $\tilde{L_1}$ and specialize $q$ to $\tilde{L_1}\cap E$. In this way $q$ specializes to $p_1$, so $\{p_1,p_2,p_3\}$ is limit of elements of $B_{12}$. This proves the first claim.

Let $\{p_1,p_2,p_3\}$ be a general point of $T$, so that $p_1,p_2\in E$, $p_3\notin E$ and $\mu(p_3)\notin L_1$. Let $q_1$ be a general point of $\tilde{L_1}\setminus E$ and let $q_2=\mu^{-1}(L_2\cap\langle \mu(p_3),\mu(q_1)\rangle)$. Then $\{q_1,q_2,p_3\}\in Y$. Now let $\langle \mu(p_3),\mu(q_1)\rangle$ move along the pencil of lines through $\mu(p_3)$ and specialize it to $\langle\mu(p_3),z\rangle$. Then $q_1$ moves along $\tilde{L_1}$ and specializes to $p_1$. In the same way $q_2$ specializes to $p_2$. Hence $\{p_1,p_2,p_3\}$ is limit of elements of $Y$.
\end{proof}
\end{lemma}

\begin{proposition}\label{cor: ter3 of del pezzo}
Let $z\in\p^2$ and let $X=\Bl_z\p^2\xrightarrow{\mu}\p^2$ be the del Pezzo surface in $\p^8$ with exceptional divisor $E$. Then
\[\Ter_3(X)=\overline{Y}\cup\overline{B_{12}}\cup\overline{B_{13}}\cup\overline{B_{23}}\]
has four components of dimension 5.
\end{proposition}

\section{Veronese varieties}
\label{section: veronese surfaces}
The purpose of this section is to describe the first nonempty Terracini loci of Veronese varieties. Intuitively we expect that, for small values of $r$, the only way for $r$ double points to give dependent conditions is to be in a very special position, with most of them lying on the same line. In  \Cref{propo: all we know about the Veronese} we confirm this expectation. Denote by $\V_n^d$ the image of the embedding $\nu_d:\p^n\to\p^{\binom{n+d}{d}-1}$ via the complete linear system of divisors of degree $d$. By \cref{def: Terracini locus with cohomology}, if $S\in(\PP^n)^{(r)}$ then $\nu_d(S)\in\Ter_r(\V_n^d)$ if and only if $h^1(I_{2S,\p^n}(d))>0$.

Thanks to \Cref{corol: rational and elliptic normal curves} we already know that $\Ter_r(\V_1^d)=\varnothing$ for every $r$, so we assume that $n\ge 2$. Moreover, \cite{AH} shows that $\V_n^2$ is $r$-defective for every relevant value of $r$. As we noted in \Cref{rmk: we assume subabundant and nondefective}, this is not very interesting, so we further assume that $d\ge 3$.

Recall that if $Z$ is a zero-dimensional scheme and $H$ is the hyperplane in $\p^n$ defined by the form $\ell$, the \emph{residue} of $Z$ with respect to $H$ is the subscheme $\Res_H(Z)\subset H$ defined by the colon ideal $H^0(I_{Z,\p^n}) : (\ell)$. For every positive integer $d$ we can consider the exact sequence of sheaves on $\p^n$

\begin{equation}\label{eq: restriction exact sequence}
0\to
I_{\Res_H(Z),\p^n}(d-1)\to
I_{Z,\p^n}(d)\to
I_{Z\cap H,H}(d)\to 0.
\end{equation}
At the level of global section, the map $H^0(I_{Z,\p^n}(d))\to H^0
(I_{Z\cap H,H }(d))$ is the restriction to $H$, while the map $H^0(I_{\Res_H(Z),\p^n}(d-1))\to H^0(
I_{Z,\p^n}(d))$ is the multiplication by $\ell$. A similar construction works not only on $\p^n$, but on any smooth variety. In \Cref{sec: SegreVeronese} we will employ it on a multi-projective space.

\begin{lemma}\label{lemma: induction for veronese terracini}
Suppose $n\geq 2$, $d\geq 3$ and let $S\subset\PP^n$ be a set of $r\geq 2$ points. Let $H\subset\PP^n$ be a hyperplane such that $S\cap H$ consists of $s\leq\frac{1}{n}\binom{n+d-1}{d}$ points in general position on $H$.  If $\nu_{d-2}(S\setminus H)\notin\Ter_{r-s}(\V_n^{d-2})$, then $\nu_d(S)\notin\Ter_r(\V_n^d)$ unless $(n-1,d,s)\in\{(2,4,5),(3,4,9),(4,3,7),(4,4,14)\}$.
\end{lemma}

\begin{proof}
We are going to apply the restriction exact sequence \eqref{eq: restriction exact sequence} with $Z=2S$. The exact sequence of sheaves
\[0\to
I_{\Res_H(2S),\p^n}(d-1)\to
I_{2S,\p^n}(d)\to
I_{2S\cap H,H}(d)\to 0\]
induces the exact sequence of vector spaces
\[
H^1(I_{\Res_H(2S),\p^n}(d-1))\to
H^1(I_{2S,\p^n}(d))\to
H^1(I_{2S\cap H,H}(d)).
\]
Consider the right hand side. Thanks to our hypothesis, $2S\cap H$ is a scheme of $s\leq\frac{1}{n}\binom{n+d-1}{d}$ double points in general position on $H$, so 
$h^1(I_{2S\cap H, H}(d))=0$ by \cite{AH}. In order to conclude, it is enough to show that $h^1(I_{\Res_H(2S),\p^n}(d-1))=0$. We apply sequence \eqref{eq: restriction exact sequence} again, with $Z=\Res_H(2S)$, and we obtain
\[0\to
I_{2(S\setminus H),\p^n}(d-2)\to
I_{\Res_H(2S),\p^n}(d-1)\to
I_{S\cap H,H}(d-1)\to 0.\]
Our hypothesis that $\nu_{d-2}(S\setminus H)\notin\Ter_{r-s}(\V_n^{d-2})$ guarantees that $h^1(I_{2(S\setminus H),\p^n}(d-2))=0$. On the right hand side, $h^1(I_{S\cap H,H}(d-1))=0$ because $S\cap H$ is a set of 
\[s\leq\frac{1}{n}\binom{n+d-1}{d}=\frac{n+d-1}{nd}\binom{n+d-2}{d-1}<\left(\frac{1}{n}+\frac{1}{d}\right)\binom{n+d-2}{d-1}<\binom{n+d-2}{d-1}\]
simple points in general position. Hence $h^1(I_{\Res_H(2S),\p^n}(d-1))=0$.
\end{proof}

\Cref{lemma: induction for veronese terracini} can be used to determine when a given set of points does not belong to the Terracini locus of a Veronese variety, but the next result can be used to say when one \textit{does} belong.

\begin{lemma}\label{lemma: veronese points on subspace}
Suppose $S\subset\PP^n$ is a set of $r$ points, $s$ of which lie on a linear subspace $P\cong\PP^m$.  If $s>\frac{1}{m+1}\binom{m+d}{d}$, then $\nu_d(S)\in\Ter_r(\V_n^d)$.
\end{lemma}

\begin{proof}
Note that $\nu_d(P)\cong \V_m^d$, and so $\nu_d(S\cap P)\in\Ter_s(\nu_d(P))$ by \cref{rmk: we assume subabundant and nondefective}.  By \cref{propos: Ter are nested} this implies that $\nu_d(S)\in\Ter_r(\nu_d(P))$, and so the result follows by \cref{propos: Ter of a subvariety}.
\end{proof}

By \cite[Theorem 1.1]{LM23} $\Ter_r(\V_n^d)=\varnothing$ if and only if $2r<d+2$. The next result shows that being collinear is the only possibility for a set of point to belong to the first nonempety Terracini locus of a Veronese variety. 

\begin{proposition}\label{proposition: first nonempty terracini}
Let $n\geq 2$ and $d\geq 3$,  and let $r=\lceil \frac{d+2}{2} \rceil$. If $S\subset \PP^n$ be a set of $r$ points that are not collinear, then $\nu_d(S)\not\in\Ter_r(\V_n^d)$.
\end{proposition}

\begin{proof}
We will prove the result by induction on $r$.  Let $r=3$.  Since $S$ is not contained in a line, it is a set of three general points and $\nu_d(S)\not\in\Ter_3(\V_n^d)$ by \cite{AH}.

Now let $r\geq 4 $. Let $H$ be an hyperplane such that $H\cap S=\{p\}$ for some $p\in S$ for which $S\setminus \{p\}$ is not contained in any line. Since $r-1=\ru{\frac{d}{2}}=\ru{\frac{(d-2)+2}{2}}$, we assume by induction hypothesis that $\nu_{d-2}(S\setminus\{p\})\not\in\Ter_{r-1}(\V_n^{d-2})$, and so the result follows by \cref{lemma: induction for veronese terracini}.
\end{proof}

Taking a step further, we focus on the next nonempty Terracini locus of a Veronese variety.  We need to treat the case of cubics separately.

\begin{proposition}\label{proposition: second nonempty terracini for cubics}
Suppose $n\geq 3$.  If $S\subset\PP^n$ is a set of 4 points not lying on a common plane, then $\nu_3(S)\not\in\Ter_4(\V_n^3)$.
\end{proposition}

\begin{proof}
By assumption we can find a hyperplane $H\subset\PP^n$ containing exactly 3 points of $S$.  Note that these 3 points are not collinear, or otherwise $S$ would be contained in a plane. This means that $S\cap H$ is in general position.  Since $\Ter_1(\V_n^1)
=\varnothing$ by \cref{rmk: we assume subabundant and nondefective} and $3\leq\frac{1}{n}\binom{n+2}{3}$ for all $n\geq 3$, the result follows by \cref{lemma: induction for veronese terracini}.
\end{proof}

\begin{proposition}\label{proposition: second nonempty terracini}
Let $n\geq 2$ and $d\ge 4$, and let $r=\lceil \frac{d+2}{2} \rceil+1$. Let $S\subset \PP^n$ be a set of $r$ points such that no $r-1$ points of $S$ are collinear. Then $\nu_d(S)\notin\Ter_r(\V_n^d)$.
\end{proposition}

\begin{proof}
We argue by induction on $d$. First suppose that $d=4$, so $S$ has four points. If $S$ is not contained in a plane, then there exists a hyperplane $H$ containing exactly 3 points of $S$, which are not collinear by assumption. In this case we proceed as in the proof of \cref{proposition: second nonempty terracini for cubics} and we conclude that $\nu_4(S)\notin\Ter_4(\V_n^4)$.  Now assume that $S$ is contained in a plane $P$. Since $S$ does not contain three collinear points, $h^1(I_{2S,P}(4))=0$ by \cite{AH}, so if $n=2$ we are done. If $n\ge 3$, then let $H_1,\dots,H_{n-2}$ be hyperplanes such that $P=H_1\cap\dots\cap H_{n-2}$. The short exact sequence of sheaves \eqref{eq: restriction exact sequence}, restricting the scheme $2S$ to the hyperplane $H_1$, gives the exact sequence of vector spaces
$$ H^1(I_{S,\p^n}(3))\to H^1(I_{2S,\p^n}(4))\to H^1(I_{2S, \p^{n-1}}(4)).
$$
On the left hand side we have $h^1(I_{S,\PP^n}(3))=0$ because for every $p\in S$ we can always find a reducible cubic containing $S\setminus\{p\}$ but not containing $p$.
In order to prove that $h^1(I_{2S,\PP^{n-1}}(4))=0$, consider again the short exact sequence \eqref{eq: restriction exact sequence}, restricting the scheme $2S$ to $H_1\cap H_2=\p^{n-2}$, and so on recursively. After $n-2$ steps we have $h^1(I_{2S,\PP^n}(4))=0$ since
$$0=H^1(I_{S,\p^3}(3))\to H^1(I_{2S,\p^3}(4))\to H^1(I_{2S, P}(4))=0.
$$
Next, suppose that $d=5$ and let $S=\{p_1,\dots,p_5\}$. We want to show that there exists a hyperplane $H\subset\p^n$ containing exactly two points of $S$, such that the three points of $S\setminus H$ are not collinear. Indeed, if no three points of $S$ are on a line, then we pick $H$ a general hyperplane containing $p_1$ and $p_2$. If $p_1,p_2,p_3$ are collinear and there is no $i\in \{1,2,3\}$ such that $p_i,p_4,p_5$ are collinear, then we pick $H$ a general hyperplane containing $p_1$ and $p_4$. If $p_1,p_2,p_3$ are collinear and there exists $i\in\{1,2,3\}$ such that $p_i,p_4,p_5$ are collinear, then our hypothesis that no 4 points of $S$ are collinear guarantees that we can pick a general hyperplane $H$ containing $\langle p_j,p_k\rangle$ with $j\in \{1,2,3\}\setminus\{i\}$ and $k\in\{4,5\}$. 
Since $\nu_3(S\setminus H) \notin \Ter_3(\V_n^3)$ by \cite{AH}, it follows that $\nu_5(S)\notin\Ter_5(\V_n^5)$ by \cref{lemma: induction for veronese terracini}.

Finally, assume that $d\geq 6$ and let $S=\{p_1,\dots,p_r\}$ be a set of $r$ points such that no  $r-1$ points of $S$ are collinear. Let $H$ be an hyperplane such that $S\cap H$ consists of one point and $S\setminus H$ contains no subset of $r-2$ collinear points. Such hyperplane always exists because even if $p_1,\ldots,p_{r-2}$ are collinear, we can take the general hyperplane through $p_1$. Since $$\#(S\setminus H)=r-1=\ru{\frac{d+2}{2}}=\ru{\frac{(d-2)+2}{2}}+1,$$
by induction hypothesis we have $\nu_{d-2}(S\setminus H)\notin\Ter_{r-1}(\V_n^{d-2})$, and so we conclude by \cref{lemma: induction for veronese terracini}.
\end{proof}

We have all the tools to prove our result on Veronese varieties.

\begin{theorem}\label{propo: all we know about the Veronese}Let $n$ and $d$ be integers such that $n\ge 2$ and $d\ge 3$ and let $N=\binom{n+d}{n}-1$. Let $\nu_d:\p^n\to 
\V_n^d\subset\p^N
$ be the $d$-Veronese embedding.  
\begin{itemize}
\item If $r=\ru{\frac{d+2}{2}}$, then $\Ter_r(\V_n^d)=\overline{\{\nu_d(S)\mid S\in (\p^n)^{(r)}\mbox{ is contained in a line}\}}$. In particular, it is irreducible of dimension $2n+r-2$.
\item $\Ter_4(\V_n^3)=\overline{\{\nu_3(S)\mid S\in (\p^n)^{(4)}\mbox{ is contained in a plane}\}}$.  In particular, it is irreducible of dimension $3n + 2$.
\item If $d\geq 4$ and $r=\ru{\frac{d+2}{2}}+1$, then $$\Ter_r(\V_n^d)=\overline{\{\nu_d(S\cup \{p\})\mid S\in (\p^n)^{(r-1)}\mbox{ is contained in a line}  \mbox{ and } p\in\p^n\setminus S\}
}.$$ In particular, it is irreducible of dimension $3n+r-3$.
\end{itemize}\end{theorem}

\begin{proof}
Set $r=\ru{\frac{d+2}{2}}$ and let us consider the first item.  One one hand, since $r>\frac{d+1}{2}=\frac{1}{1+1}\binom{1+d}{d}$, we have $\nu_d(S)\in\Ter_r(\V_n^d)$ if $S$ is contained on a line by \cref{lemma: veronese points on subspace}.  On the other hand, \cref{proposition: first nonempty terracini} tells us that putting all points on a line is the only way for $\nu_d(S)$ to be in the Terracini locus.

Now suppose $d=3$ and $r=4$.  Since $4>\frac{10}{3}=\frac{1}{2+1}\binom{2+3}{3}$, we see that if all 4 points of $S$ are contained in a plane, then $\nu_3(S)\in\Ter_4(\V_n^3)$ by \cref{lemma: veronese points on subspace}.  %It also tells us that $\nu_3(S)\in\Ter_4(\V_n^3)$ if at least 3 points of $S$ are on a line, but this would force all of $S$ to be on a plane.  
On the other hand, there are no other sets of points in the Terracini locus by \cref{proposition: second nonempty terracini for cubics}.  Note that this means that $\Ter_4(\V_2^3)=(\V_2^3)^{(4)}$, but this is expected by \cref{rmk: we assume subabundant and nondefective}.

Let us focus now on the third item and set $r=\ru{\frac{d+2}{2}}+1$. On one hand, since $r-1>\frac{1}{1+1}\binom{1+d}{d}$,  we have $\nu_d(S)\in\Ter_r(\V_n^d)$ if $r-1$ points of $S$ lie on a line by \cref{lemma: veronese points on subspace}.  On the other hand, \cref{proposition: second nonempty terracini} shows that the general element of $\Ter_r(\V_n^d)$ is of this form.
\end{proof}

Since we know the dimension of the first nonempty Terracini loci, \Cref{propos: Ter are nested} allows us to bound the dimension of higher Terracini loci of the Veronese variety.
\begin{corollary}
Let $n\ge 2$ and $d\ge 3$. Suppose that $r$ is an integer satisfying \eqref{eq: subabundant} such that $2r\ge d+2$. Then $$\dim\Ter_r(\V_n^d)\ge 2n+\ru{\frac{d+2}{2}}-2+n\left(r-\ru{\frac{d+2}{2}}\right).$$
In the special case of cubics, we have
$$\dim\Ter_r(\V_n^3)\geq 3n + 2 + n(r - 4)$$
for every $r\ge 4$ satisfying \cref{eq: subabundant}.
\end{corollary}

\section{Segre-Veronese varieties}\label{sec: SegreVeronese}
In this section we characterize the first nonempty Terracini loci of Segre-Veronese varieties. 
The behaviour just analyzed in the Veronese case gives us some intuition about the Terracini loci of Segre-Veronese varieties for the smallest possible values of $r$. We might guess that for $r$ double points to give dependent conditions, they must be in the very special position where all of them belong to the same ruling.  In \cref{theorem: first nonempty SV}, we will prove that this is indeed usually the case. 

Let $n_1,\ldots,n_k$ and $d_1,\ldots,d_k$ be positive integers. Let $N=\binom{n_1+d_1}{n_1}\cdot\ldots\cdot\binom{n_k+d_k}{n_k}-1$ and let 
$$\nu:\p^{n_1}\times\cdots\times\p^{n_k}\rightarrow X=\SV^{d_1,\ldots,d_k}_{n_1\times \cdots \times n_k}\subset\p^N$$ be the Segre-Veronese embedding.  For $i\in\{1,\ldots,k\}$, denote by $\pi_i:\p^{n_1}\times\cdots\times\p^{n_k}\rightarrow \PP^{n_i}$ the projection onto the $i$th factor. By \cite[Theorem 1.1]{LM23}, the smallest integer $r$ for which $\Ter_r(X)\neq \varnothing $ is $r=\min_i\ru{\frac{d_i+2}{2}}$.

\begin{remark}\label{remark: no segre subvarieties}
Just as in the beginning of \Cref{section: veronese surfaces}, we make some assumptions in order to avoid defective cases. Indeed, if there exist distinct $i,j$ for which $d_i=d_j=1$, then $X$ contains a copy of the Segre variety $\SV_{n_1\times n_2}^{1,1}$, which is $r$-defective for every $r$ satisfying \cref{eq: subabundant}, therefore \Cref{propos: Ter of a subvariety} implies that the Terracini locus is larger than the one we described in our guess.  For this reason we assume that $\#\{i\mid d_i=1\}\leq 1$, and so in particular we will not discuss the Segre varieties with multidegree $(1,\ldots,1)$.   We refer the interested reader to \cite{BBS} for first results on Terracini loci of Segre varieties.
\end{remark}
 Since it will come in handy later, let us recall the shape of the tangent space of a Segre-Veronese variety at a point $[p]\in \mathrm{SV}^{d_1,\dots,d_k}_{n_1\times \cdots \times n_k}$. We represent affinely $[p]\in \mathrm{SV}^{d_1,\dots,d_k}_{n_1\times \cdots \times n_k}$ as $p= L_1^{d_1}\cdots L_k^{d_k}$, where $L_i\in \mathbb{K}[x_{i,0},\ldots,x_{i,n_i}]_1$. Note that many authors would represent this as $p=L_1^{d_1}\otimes\cdots\otimes L_k^{d_k}$, but we omit the $\otimes$'s for brevity.  Then
\begin{align*}
    T_{[p]}\mathrm{SV}^{d_1,\dots,d_k}_{n_1\times \cdots \times n_k}= &\mathbb{P} \left( 
    \sum_{i=1}^k \prod_{j\neq i} L_j^{d_j}L_i^{d_i-1}\mathbb{K}[x_{i,0},\dots,x_{i,n_i}]_1\right) \\
    =& \mathbb{P} \left( L_1^{d_1 -1}L_2^{d_2}\cdots L_k^{d_k}  \mathbb{K}[x_{1,0},\dots,x_{1,n_1}]_1 + \cdots + L_1^{d_1}\cdots L_k^{d_{k-1}}L_k^{d_k-1}  \mathbb{K}[x_{k,0},\dots,x_{k,n_k}]_1  \right).
\end{align*}
For all $i\in\{1,\dots,k\}$, let $Y_{i,1},\ldots, Y_{i,n_i} \in \mathbb{K}[x_{i,0},\ldots,x_{i,n_i}]_1 $ be such that $\{ L_i, Y_{i,1},\ldots, Y_{i,n_i}\}$ is a basis of $\mathbb{K}[x_{i,0},\ldots,x_{i,n_i}]_1 $. A set of generators for $T_{[p]}\mathrm{SV}^{d_1,\dots,d_k}_{n_1\times \cdots \times n_k}$ is $A_1\cup \cdots \cup A_k $, where
\begin{align*}
   A_1=&\left\{L_1^{d_1-1}L_1L_2^{d_2}\cdots L_k^{d_k}  ,L_1^{d_1-1}Y_{1,1}L_2^{d_2}\cdots L_k^{d_k},\dots,  L_1^{d_1-1}Y_{1,n_1}L_2^{d_2}\cdots L_k^{d_k} \right\},\\
    & \vdots \\
    A_k=&\left\{ L_1^{d_1}\cdots L_{k-1}^{d_{k-1}}L_k^{d_k-1}L_k, L_1^{d_1}\cdots L_{k-1}^{d_{k-1}}L_k^{d_k-1}Y_{k,1}, \dots, L_1^{d_1}\cdots L_{k-1}^{d_{k-1}}L_k^{d_k-1}Y_{k,n_k}\right\}.
\end{align*}
Notice that $A_1\cup \cdots \cup A_k $ is not linearly independent, because every $A_i$ contains $L_1^{d_1}\cdots L_k^{d_k} $. If we remove it from $A_2,\dots,A_k$, then the remaining elements of $A_1\cup \cdots \cup A_k$ would form a basis of the affine tangent space.

The following analog of \cref{lemma: induction for veronese terracini} to Segre-Veronese varieties will be useful for several results in this section.  At the cost of having a much smaller upper bound on the number of points lying in the chosen hyperplane% than in the Veronese result
, it has the advantage of not requiring these points to be in general position.

\begin{lemma}\label{lemma: induction for segre-veronese terracini}
Suppose that $d_i\geq 3$ for some $i\in\{1,\ldots,k\}$ and $H'\subset\PP^{n_i}$ is a hyperplane.  Let $H=\pi_i^{-1}(H')$ and suppose $S\subset\PP^{n_1}\times\cdots\times\PP^{n_k}$ is a set of $r$ points such that $S\cap H$ consists of $s$ points with $s\leq\ru{\frac{d_j}{2}}$ for all $j$.  If $\nu(S\setminus H)\notin\Ter_{r-s}(\SV_{n_1\times\cdots\times n_k}^{d_1,\ldots,d_i-2,\ldots,d_k})$, then $\nu(S)\notin\Ter_r(\SV_{n_1\times\cdots\times n_k}^{d_1,\ldots,d_k})$.
\end{lemma}

\begin{proof}
We follow the proof of \cref{lemma: induction for veronese terracini}.  To simplify notation, let $\PP^{\mathbf n}=\PP^{n_1}\times\cdots\times\PP^{n_k}$, $\mathbf d = (d_1,\ldots,d_k)$,  $\mathbf e_i=(0,\ldots,0,1,0,\ldots,0)$ with the 1 in the $i$th position, and $\mathbf 1 = (1,\ldots,1)$.

Applying \cref{eq: restriction exact sequence} and taking cohomology, we have the following exact sequence of vector spaces
\begin{equation*}
H^1(I_{\Res_H(2S),\PP^{\mathbf n}}(\mathbf d-\mathbf e_i))\longrightarrow H^1(I_{2S,\PP^{\mathbf n}}(\mathbf d))\longrightarrow H^1(I_{2S\cap H,H}(\mathbf d)).
\end{equation*}

The vector space on the right-hand side is zero due to the upper bound on $s$ by \cite[Theorem 1.1]{LM23}, and so it remains to consider the left-hand side.  Applying \cref{eq: restriction exact sequence} again, we have
\begin{equation*}
H^1(I_{2(S\setminus H),\PP^{\mathbf n}}(\mathbf d-2\mathbf e_i))\longrightarrow H^1(I_{\Res_H(2S),\PP^{\mathbf n}}(\mathbf d-\mathbf e_i))\longrightarrow H^1(I_{S\cap H,H}(\mathbf d-\mathbf e_i)).
\end{equation*}

The left-hand side is zero by hypothesis, and so it remains to consider the right-hand side, the first cohomology space of a scheme of $s$ simple points on $H$.  To show that it is zero, it is sufficient to prove that for every $p\in S\cap H$ there exists a divisor in $H^0(\oo_{H}(\mathbf d - \mathbf e_i))$ that contains every point of $S\cap H$ except $p$.  Indeed, choose $p\in S\cap H$ and pick a general divisor $D_p\in H^0(\oo_H(\mathbf 1))$ through $p$. Let $E\in H^0(\oo_H(\mathbf d - \mathbf e_i + (1 - s)\mathbf 1))$ be a general divisor, noting that the elements of the multidegree of $E$ are all positive due to our bound on $s$.  Finally, $S\cap H\setminus\{p\}\in\sum_{q\in S\cap H\setminus\{p\}}D_q + E$, as desired.
\end{proof}

\begin{definition}\label{def: r and J for SV}
In the rest of this section, given a multidegree $(d_1,\ldots,d_k)$, we will always assume that $r$ is the smallest value for which the $r$-Terracini locus of the Segre-Veronese variety is nonempty by \cite[Theorem 1.1]{LM23}, that is
\begin{equation*}
    r = \min\left\{\ru{\frac{d_i+2}{2}}\middle| i\in\{1,\ldots,k\}\right\}.
\end{equation*}

We will frequently isolate the indices of the $d_i$'s that determine $r$, and so we define
\begin{equation*}
    J = \left\{i\in\{1,\ldots,k\}\middle|\ru{\frac{d_i+2}{2}}= r\right\}.
\end{equation*}
\end{definition}
This set $J$ will be relevant for all the results of this part. All factors of the multiprojective space indexed by elements of $J$ will provide a component of the first nonempty Terracini locus. As we will see, %for all $ i\in J$, 
such a component will be formed by all $r$ points of the multiprojective space whose projection on the $i$th factor are collinear and whose projection on the other factors is just a point. First, we prove that points with this structure belong to the first nonempty Terracini locus.

\begin{proposition}\label{pro: SV has nonempty terracini}
Let $d_1,\ldots,d_k$ be positive integers satisfying \cref{remark: no segre subvarieties}. Let $X=\SV^{d_1,\ldots,d_k}_{n_1\times \cdots \times n_k}$ and let
$S\subset \p^{n_1}\times\cdots\times\p^{n_k}$ be a set of $r$ points. Assume that there exists $i\in J$ such that {$\pi_i(S)$} is collinear and $\pi_\ell(S)$ is a point for every $\ell\neq i$. Then $\nu(S)\in\Ter_r(X)$.
\end{proposition}
\begin{proof}
By permuting the factors if necessary, we assume that $i=1$. For every $\ell\in\{2,\ldots,k\}$, let $q_\ell\in\p^{n_\ell}$ be the point $\pi_\ell(S)$. Assume first that $n_1\geq 2$ and let $$Y=\nu(\p^{n_1}\times \{ q_2\}\times \cdots \times \{q_k\}) \subset X.$$
Then $\nu_{|Y}$ is the $d_1$-Veronese embedding of $\p^{n_1}$. We want to show that $\nu(\pi_1(S)\times\{ q_2\}\times \cdots \times \{q_k\})\in\Ter_r(Y)$. If $d_1=2$, then this is true because $r=2$ and $\V_{n_1}^2$ is 2-defective by \cite{AH}, therefore $\Ter_2(Y)=Y^{(2)}$. If $d_1\ge 3$, it is true by \Cref{propo: all we know about the Veronese}. Therefore $\nu(\pi_1(S)\times\{ q_2\}\times \cdots \times \{q_k\})\in\Ter_r(X)$ by \Cref{propos: Ter of a subvariety}.

Now assume that $n_1=1$. We can no longer use \Cref{propos: Ter of a subvariety} because the rational normal curve $\V_1^{d_1}$ has empty Terracini loci, as seen in \Cref{corol: rational and elliptic normal curves}. In this case we will prove the result by looking at the span of the affine tangent spaces associated to the points of $\nu(S)$. Take two distinct linear forms $[M_0],[M_1]\in \mathbb{K}[x_{1,0},x_{1,1}]$. For all $\ell\in\{2,\ldots,k\}$ let $[L_\ell]\in \mathbb{K}[x_{\ell,0},\dots,x_{\ell,n_{\ell}}]_1 $. We represent affinely the $r$ points of $\nu(S)$ as
\[\begin{tabular}{cc}
$p_\alpha=M_0^{d_1}L_2^{d_2}\cdots L_k^{d_k}$, &
$p_\beta=M_1^{d_1}L_2^{d_2}\cdots L_k^{d_k}$
\end{tabular} \]
and $p_s=(\alpha_sM_0+\beta_sM_1)^{d_1}L_2^{d_2}\cdots L_k^{d_k}$
for each  $s\in\{3,\dots,r\}$. Then
\begin{align*}
    A_{1}=\left\{   M_0^{d_1}L_2^{d_2}\cdots L_k^{d_k},M_0^{d_1-1}M_1L_2^{d_2}\cdots L_k^{d_k}  \right\}
    \end{align*}
is a set of two independent tangent vectors to $X$ at $[p_\alpha]$, so we can complete $A_1$ to an affine basis $A_{\alpha}$ of $T_{[p_\alpha]}X$. In a similar way, we can complete 
\[B_{1}=\left\{   M_1^{d_1}L_2^{d_2}\cdots L_k^{d_k},M_1^{d_1-1}M_0L_2^{d_2}\cdots L_k^{d_k}  \right\}
\]
to a basis $B_{\beta}$ of the affine tangent space to $X$ at $[p_\beta]$. 
Now we do the same for $p_s$ for each $s\in\{3,\dots,r\}$. If we call
     \begin{align*}
    C_{s,1}=&\left\{   (\alpha_sM_0+\beta_s M_1)^{d_1-1}M_0L_2^{d_2}\cdots L_k^{d_k},(\alpha_sM_0+\beta_s M_1)^{d_1-1}M_1L_2^{d_2}\cdots L_k^{d_k}  \right\}
    \end{align*}
then we can complete $C_{s,1}$ to a basis $C_s$ of the affine tangent space to $X$ at $[p_{s}]$.
In order to prove that $\nu(S)\in\Ter_r(X)$, we have to show that the set $A_{\alpha}\cup B_{\beta}\cup C_3\cup \cdots \cup C_r $ given by the union of all bases is linearly dependent. It is enough to show that $A_1\cup B_1\cup C_{3,1}\cup\cdots\cup C_{r,1}$ is linearly dependent. Observe that the $2r$ elements of $A_1\cup B_1\cup C_{3,1}\cup\cdots\cup C_{r,1}$ are of the form $f\cdot L_2^{d_2}\cdots L_k^{d_k}$ for some $f\in\kk[x_{1,0},x_{1,1}]_{d_1}$. Since $2r\ge d_1+2>\dim \kk[x_{1,0},x_{1,1}]_{d_1}$, they are linearly dependent.
\end{proof}

We want to prove the converse of \Cref{pro: SV has nonempty terracini} by induction on $r$. Notice that the case  $r=2$, or equivalently $\min\{d_1,\dots,d_k\}\in\{1,2\}$, is different from the other ones, because  two points in $\p^{n_i}$ are always collinear. In other words, if $r=2$ then the first part of the hypothesis of \Cref{pro: SV has nonempty terracini} is always satisfied. For this reason, we treat this case separately in \Cref{prop: two points SV}.  Note that this is the only case where we need to worry about \cref{remark: no segre subvarieties}.

\begin{proposition}\label{prop: two points SV}
Let $d_1,\ldots,d_k$ be positive integers satisfying \cref{remark: no segre subvarieties} for which $r=2$. Let $X=\SV^{d_1,\ldots,d_k}_{n_1\times \cdots \times n_k}$ and let $S\subset \PP^{n_1}\times \cdots \times \PP^{n_k}$ be a set of two points.
Then $\nu(S)\in \Ter_2(X)$ 
if and only if there exists $i\in J$ such that $\#\pi_\ell(S)=1$ for all $\ell\in\{1,\dots,k\}\setminus\{i\}$.
\end{proposition}
\begin{proof}
Represent affinely the two points of $\nu(S)$  as  
$$p_1=L_1^{d_1}\cdots L_k^{d_k}\mbox{ and } p_2=M_1^{d_1}\cdots M_k^{d_k},$$
where $[L_\ell],[M_\ell]\in \mathbb{K}[x_{\ell,0},\dots,x_{\ell,n_{\ell}}]_1$ for all $\ell\in\{1,\dots,k\}$. Let $i$ be an index of a factor with $\pi_i([p_1])\neq\pi_i([p_2])$ for which $d_i$ is minimal.
Since $\#\pi_i(S)=2$ we know that $L_i$ and $M_i$ are linearly independent forms. Let $Y_{i,2},\dots,Y_{i,n_i}\in \kk[x_{i,0},\dots ,x_{i,n_i}]_1$ be such that $\{L_i,M_i, Y_{i,2},\dots,Y_{i,n_i}\}$ is a basis of $ \kk[x_{i,0},\dots ,x_{i,n_i}]_1$. We call $A_1\cup \cdots \cup A_k$  an affine basis for $T_{[p_1]}X$ and $B_1\cup \cdots \cup B_k$ an affine basis for $T_{[p_2]}X$ constructed as explained at the beginning of this section. We first notice that for all $\ell\in \{ 1,\dots,k\} \setminus \{ i\}$ every element of $A_\ell$ is multiplied by $L_i^{d_i}$ and any element of $B_\ell$ is multiplied by $M_i^{d_i}$, where $[L_i]\neq [M_i]$. Therefore, the elements of $A_\ell$ and $B_\ell$ are linearly independent for all $\ell\in \{ 1,\dots,k\} \setminus \{ i\}$. On the $i$th factor we have the following independent generators for $p_1$ and $p_2$ respectively: 
    \begin{align}\label{eq: 2points tg spaces ith factor}
        A_i=\left\{ \prod_{\ell=1}^k L_\ell^{d_\ell},   L_i^{d_i-1}M_i\prod_{\ell\neq i}L_\ell^{d_\ell}, L_i^{d_i-1}Y_{i,e}\prod_{\ell\neq i}L_\ell^{d_\ell} 
        \right\}_{e=2,\dots,n_i},\\ \nonumber
       B_i=\left\{ \prod_{\ell=1}^k M_\ell^{d_\ell},   M_i^{d_i-1}L_i\prod_{\ell\neq i}M_\ell^{d_\ell}, M_i^{d_i-1}Y_{i,e}\prod_{\ell\neq i}M_\ell^{d_\ell} 
        \right\}_{e=2,\dots,n_i}.
    \end{align}
Assume for the moment $i\notin J$. In this case we need to prove that $\nu(S)\notin\Ter_2(X)$.  Since $i\notin J$, we have $d_\ell\geq 3$ for all $\ell$ and therefore also $A_i$ and $B_i$ are linearly independent.

Assume now $i\in J$. In this case $\nu(S)\in\Ter_2(X)$ if and only if $\#\pi_\ell(S)=1$ for all $\ell\in\{1,\dots,k\}\setminus\{i\}$. Thanks to  \Cref{pro: SV has nonempty terracini}, we only have to prove that if $ \nu(S)\in \Ter_2(X)$ then $\#\pi_\ell(S)=1$ for all $\ell\neq i$. Let us focus on the second element of $A_i$ and $B_i$ in \eqref{eq: 2points tg spaces ith factor}. Since $i\in J$ the degree $d_i\in \{1,2 \}$ but the second element of each set $A_i,B_i$ are linearly independent if there exists at least one $[M_\ell]\neq [L_\ell]$. {Therefore, if $\#\pi_\ell(S)=2$ for some $\ell \neq i$ then $A_i$ and $B_i$ are independent and hence $\nu(S)\notin \Ter_2(X)$.} Note that possibly $A_l\cap B_l\neq \emptyset$ if there exists $\ell\neq i$ for which $d_\ell=1$, however, this contradicts the hypothesis from \cref{remark: no segre subvarieties}.
\end{proof}

Now that the case $r=2$ is settled, we consider $r\ge 3$, or equivalently, $d_i\geq 3$ for all $i\in\{1,\ldots,k\}$. The next result proves that if a set of $r$ points $S\subset \PP^{n_1}\times \cdots \times \PP^{n_k}$ is not collinear when projected on one of the factors indexed by an element of $J$, then its image via the Segre-Veronese embedding is not an element of the $r$-Terracini locus.

\begin{lemma}\label{lemma: if not collinear before j hat then not in the terracini}
Let $X=\SV^{d_1,\ldots,d_k}_{n_1\times \cdots \times n_k}$ and let $S\subset \PP^{n_1}\times\cdots \times \PP^{n_k}$ be a set of $r\geq 3$ points. If there exists an $i\in J$ such that $\#\pi_\ell(S)=1$ for all $\ell\neq i$ but $\pi_i(S) $ is not collinear, then $\nu(S)\notin \Ter_r(X)$.
\end{lemma}
\begin{proof}
Up to permuting the factors, it is not restrictive to prove the result for $i=1$. Notice that that our assumption implies that $n_1\ge 2$. We argue by induction on $r$. Assume that  $r=3$ and choose coordinates such that
\[\begin{matrix}
p_1=J_0^{d_1}L_2^{d_2}\cdots L_k^{d_k}, & & p_2=J_1^{d_1}M_2^{d_2}\cdots M_k^{d_k}, &  & p_3=J_2^{d_1}P_2^{d_2}\cdots P_k^{d_k},
\end{matrix}
    \]
for some linear forms $L_i,M_i,P_i\in\kk[x_{i,0},\dots,x_{i,n_i}]_1$. Choose $Y_{1,1},\dots,Y_{1,n_1}\in \kk[x_{1,0},\dots,x_{1,n_1}]_1$ such that $\{ J_i, Y_{1,1},\dots,Y_{1,n_1}\}$ is a basis of $\kk[x_{1,0},\dots,x_{1,n_1}]_1 $ for all  $i\in\{1,2,3\}$. For every $s\in\{2,\dots,k\}$, choose linear forms $Y_{s,1},\ldots,Y_{s,n_s}\in\kk[x_{s,0},\dots,x_{s,n_s}]_1$ such that $\{L_s,Y_{s,1},\ldots,Y_{s,n_s}\}$, $\{M_s,Y_{s,1},\ldots,Y_{s,n_s}\}$ and $\{P_s,Y_{s,1},\ldots,Y_{s,n_s}\}$ are three bases of $\kk[x_{s,0},\dots,x_{s,n_s}]_1$. Call 
\begin{align*}
A_1=&\left\{J_0^{d_1} \prod_{t=2}^kL_t^{d_t} ,J_0^{d_1-1}Y_{1,e}\prod_{t=2}^kL_t^{d_t} \right\}_{e=1,\ldots,n_1}, &A_s& =\left\{J_0^{d_1}\prod_{t\neq s,1}L_t^{d_t}Y_{s,j} \right\}_{j=1,\ldots,n_s},
\\
B_1=&\left\{ J_1^{d_1}\prod_{t=2}^kM_t^{d_t} ,J_1^{d_1-1}Y_{1,e}\prod_{t=2}^kM_t^{d_t} \right\}_{e=1,\ldots,n_1}, &B_s& =\left\{J_1^{d_1}\prod_{t\neq s,1}M_t^{d_t}Y_{s,j} \right\}_{j=1,\ldots,n_s},
\\
C_1=&\left\{J_2^{d_1}\prod_{t=2}^kP_t^{d_t}, J_2^{d_1-1}Y_{1,e}\prod_{t=2}^kP_t^{d_t} \right\}_{e=1,\ldots,n_1}, &C_s& =\left\{J_2^{d_1}\prod_{t\neq s,1}P_t^{d_t}Y_{s,j} \right\}_{j=1,\ldots,n_s},
\end{align*}
for every $s\in\{2,\dots,k\}$. Then $A_1\cup\cdots\cup A_k$, $B_1\cup\cdots\cup B_k$ and $C_1\cup\cdots\cup C_k$ are bases for the affine tangent spaces at $p_1$, $p_2$ and $p_3$ respectively. Since $J_0,J_1,J_2$ are linearly independent, $A_s\cup B_s\cup C_s$ is linearly independent for each $s\in\{2,\ldots,k\}$. By hypothesis
$d_1-1\geq2$, so $A_1\cup B_1\cup C_1$ is also independent. This concludes the base case because the elements of $A_h,B_h,C_h$ belong to complementary spaces with respect to $A_g,B_g,C_g$ if $h\neq g$.

Now assume that $r\geq 4$.  Since we are dealing with more than three points, there exists $H\in H^0(\oo_{\p^{n_1}\times\cdots\times\p^{n_k}}(1,0,\ldots,0))$ such that $\pi_1(S\cap H)$ consists of a unique point  and $\pi_1(S\setminus H)$ is still not collinear.  By our assumption that $\pi_\ell(S)$ is a point for all $\ell\neq 1$, it follows that $S\cap H$ also consists of a unique point.  Since $r - 1=\ru{\frac{d_1}{2}}=\ru{\frac{(d_1-2)+2}{2}}$, we assume by induction that $\nu(S\setminus H)\notin\Ter_{r-1}(\SV_{n_1\times\cdots\times n_k}^{d_1-2,d_2,\ldots,d_k})$, and then apply \cref{lemma: induction for segre-veronese terracini}.
\end{proof}

\begin{lemma}\label{lemma: step 2 unified }
Suppose $X=\SV^{d_1,\ldots,d_k}_{n_1\times \cdots \times n_k}$ and let $S\subset \PP^{n_1}\times 	\cdots \times \PP^{n_k}$ be a set of $r\geq 3$ points. If  there exists $\ell\notin J$ such that $\#\pi_\ell(S)\geq2$ then $ \nu(S)\notin \Ter_r(X)$.
\end{lemma}
\begin{proof}
Without loss of generality, we assume that $\ell=k$.  For a given point $a\in\pi_k(S)$ and a general hyperplane $H'\subset\PP^{n_k}$ containing $a$, let $H=\pi_k^{-1}(H')$ and note that $S\cap H$ may contain multiple points, all with $a$ in the $k$th factor.  However, at least one point of $S$ does not lie in $H$, and so if $s=\#(S\cap H)$, then $s\leq r-1 = \ru{\frac{d_1}{2}}$ and thus $\#(S\setminus H)=r-s\leq r-1$ for all $i$.  Since $k\notin J$, it follows by \cref{def: r and J for SV} that $r-s\leq\ru{\frac{d_k-2}{2}}$ as well, and so $\nu(S\setminus H)\notin\Ter_{r-s}(\SV_{n_1\times\cdots\times n_k}^{d_1,\ldots,d_k-2})$ by \cite[Theorem 1.1]{LM23}.  Therefore, $\nu(S)\notin\Ter_r(X)$ by \cref{lemma: induction for segre-veronese terracini}.
\end{proof}

\begin{lemma}\label{lemma: third lemma converse SV}
Suppose $X=\SV^{d_1,\ldots,d_k}_{n_1\times \cdots \times n_k}$ and let  $S\subset \PP^{n_1}\times\cdots \times \PP^{n_k}$ be a set of $r\geq 3$ points such that there exist distinct $i,\ell \in J$ such that $\pi_i(S)\geq 2$ and $\pi_{\ell}(S)\geq 2$. Then $\nu(S)\notin \Ter_r(X) $.
\end{lemma}
\begin{proof}
For simplicity we set $\ell=1$.  For a given point $a\in\PP^{n_1}$, we can take a general hyperplane $H'\subset\PP^{n_1}$ through $a$ and let $H=\pi_1^{-1}(H')$.  However, we choose $a$ carefully as follows depending on the value of $\#\pi_1(S)$.

First, assume $\#\pi_1(S)<r$.  This means that there is an element of $\pi_1(S)$, say $a$, that corresponds to at least two points of $S$ which differ in another factor (possibly the $i$th).  Take a general $H$ passing through $a$, so if $s=\#(S\cap H)$, then $s\geq 2$ and $\#(S\setminus H)=r-s\leq r-2=\ru{\frac{d_1-2}{2}}$.  Therefore $\nu(S\setminus H)\notin\Ter_{r-s}(\SV_{n_1\times\cdots\times n_k}^{d_1-2,d_2,\ldots,d_k})$ by \cite[Theorem 1.1]{LM23}.  And since $s\leq\#\pi_1(S)\leq r - 1=\ru{\frac{d_1}{2}}$, we see that $\nu(S)\notin\Ter_r(X)$ by \cref{lemma: induction for segre-veronese terracini}.

Now assume that $\#\pi_1(S)=r$, i.e., every point of $S$ has a distinct element in the first factor.  Since $\#\pi_i(S)\geq 2$, we may select distinct $b_1,b_2\in\pi_i(S)$, and thus there must exist distinct $a_1,a_2\in\pi_1(S)$ such that $a_1\in\pi_1(\pi_i^{-1}(b_1))$ and $a_2\in\pi_1(\pi_i^{-1}(b_2))$.  But since $r\geq 3$, we may find yet another element $a_3\in\pi_1(S)$ and choose a general hyperplane $H$ through $a_3$.  Then by construction, $\{a_1,a_2\}\subseteq\pi_1(S\setminus H)$ and $\{b_1,b_2\}\subseteq\pi_i(S\setminus H)$.

We argue by induction on $r$.  For $r=3$, note that $\nu(S\setminus H)\notin\Ter_2(\SV_{n_1\times\cdots\times n_k}^{d_1-2,d_2,\ldots,d_k})$ by \cref{prop: two points SV} since at least two factors have projections of cardinality at least two, and so $\nu(S)\notin\Ter_3(X)$ by \cref{lemma: induction for segre-veronese terracini}.  For $r\geq 4$, we may assume by induction using \cref{lemma: induction for segre-veronese terracini} that $\nu(S\setminus H)\notin\Ter_{r-1}(\SV_{n_1\times\cdots\times n_k}^{d_1-2,d_2,\ldots,d_k})$ since $r-1=\ru{\frac{d_1-2+2}{2}}$, and so $\nu(S)\notin\Ter_r(X)$.
\end{proof}

We are now ready to prove the converse of \Cref{pro: SV has nonempty terracini} for all $r\geq 3$.
\begin{proposition}
Let $X=\SV^{d_1,\ldots,d_k}_{n_1\times \cdots \times n_k}$ and let  $S\subset \PP^{n_1}\times 	\cdots \times \PP^{n_k}$ be a set of $r\geq 3$ points such that for every $i\in J$,  at least one of the following holds:
\begin{itemize}
    \item $\pi_i(S)$ is not collinear,
\item there exists $\ell\in\{1,\ldots,k\}\setminus\{i\}$ such that $\#(\pi_\ell(S))\ge 2$.
\end{itemize}
Then $\nu(S)\notin \Ter_r(X)$.
\end{proposition}
\begin{proof}
   If there exists an $ i\in J$ such that $\pi_i(S)$ is not collinear, then $\nu(S)\notin\Ter_r(X)$ by \Cref{lemma: if not collinear before j hat then not in the terracini}. Therefore we assume that for every $i\in J$ the set $\pi_i(S)$ is collinear and there exist two distinct indices $\ell\neq i$ such that $\#(\pi_\ell(S))\ge 2$. If there is some $\ell\notin J$ such that $\#\pi_\ell(S)\geq 2$, then $\nu(S)\notin\Ter_r(X)$ by \Cref{lemma: step 2 unified }. Hence we only have to prove our statement under the assumption that there exist distinct $i,\ell \in J$ such that $\#\pi_i(S)\ge 2$ and $\#\pi_\ell(S)\geq 2$, which is done in \Cref{lemma: third lemma converse SV}.
\end{proof}

Combining the above result with \Cref{pro: SV has nonempty terracini} we obtain our characterization of the first nonempty Terracini loci of Segre-Veronese varieties.

\begin{theorem}\label{theorem: first nonempty SV}
Let $d_1,\ldots,d_k$ be positive integers such that $\#\{i\mid d_i=1\}\leq 1$. Let $X=\SV_{n_1\times\cdots\times n_k}^{d_1,\dots,d_k}$ be the Segre-Veronese variety and let $r$ and $J$ be as in \cref{def: r and J for SV}. For $i\in\{1,\ldots,k\}$, define 
    \begin{align*}
\mathcal{T}_i &=
\{S\in (\p^{n_1}\times\cdots\times\p^{n_k})^{(r)}\mid\dim\langle\pi_i(S)\rangle=1\mbox{ and } \pi_j(S)\mbox{ is a point }\forall\ j\neq i \}
.\end{align*}
Then $$\Ter_r(X)=\bigcup_{i\in J}\overline{\{\nu(S)\mid S\in\mathcal{T}_i\}}$$
has an irreducible component of dimension $(n_1+\cdots +n_k)+ n_i+r-2$ for each $i\in J$.
\end{theorem}

\Cref{theorem: first nonempty SV} gives a complete characterization of the first nonempty Terracini locus of a Segre-Veronese variety. 
By applying \Cref{propos: Ter are nested}, we get a lower bound on the dimension of higher Terracini loci of Segre-Veronese varieties.
\begin{corollary}\label{corollary: higher SV terracini loci}
Let $d_1,\ldots,d_k$ be positive integers satisfying \cref{remark: no segre subvarieties} and let $X=\SV_{n_1\times\cdots\times n_k}^{d_1,\ldots,d_k}$ be the Segre-Veronese variety. Let $s\geq r$ be an integer satisfying \eqref{eq: subabundant}. Then $\Ter_s(X)$ has dimension at least
\[(n_1+\cdots+n_k)(s-r)+\max_{i\in J}\{(n_1+\cdots +n_k)+ n_i+r-2\}.\]
\end{corollary}

We conclude this section by highlighting a link between Terracini loci of Segre-Veronese surfaces and unexpected curves. 

\subsection*{A connection to unexpected curves} A standard tool to deal with Segre-Veronese varieties is the \emph{multiprojective-affine-projective} method, that allows to work on linear systems on $\p^{n_1}\times\p^{n_2}$ by studying linear systems on $\p^{n_1+n_2}$. We recall here a special case of \cite[Theorem 1.5]{CGG05}.
\begin{lemma}\label{lemma: multiproj-affine-proj} Let $f:\p^1\times\p^1\dashrightarrow\p^2$ be the birational map defined by
\[([s_1:t_1],[s_2:t_2])\mapsto
\left[1:\frac{t_1}{s_1}:\frac{t_2}{s_2}\right].
\]
Let $q_1=[0:1:0]$ and $q_2=[0:0:1]$ be the two indeterminacy points for $f^{-1}$. If $S$ is a set of distinct points and $Q=cq_1+dq_2\subset\p^2$ is a scheme of two fat points of multiplicities $c$ and $d$ supported at $q_1,q_2$, then $h^0(I_{2S,\p^1\times\p^1}(c,d))=h^0(I_{Q+2f(S),\p^2}(c+d))$.
\end{lemma}
We want to draw a connection between the Terracini locus of Segre-Veronese surfaces and \emph{unexpected curves}, introduced in \cite[Definition 2.1]{unexpected}. A set of points $Z\subset\p^2$ admits an unexpected curve in degree $d$ if
\[h^0(I_{Z+(d-1)p}(d))>\max\left\{h^0(I_Z(d))-\binom{d}{2},0\right\}\mbox{ for every } p\in\p^2.\]
Now we want to show that sets of points admitting unexpected curves can be examples of elements of the Terracini loci of $\SV_{1\times 1}^{1,d-1}$.
\begin{proposition}\label{propos: terracini loci and unexpected curves}
Fix the birational map $f:\p^1\times\p^1\dashrightarrow\p^2$ and the points $q_1,q_2\in\p^2$ defined in \Cref{lemma: multiproj-affine-proj}. Consider the Segre-Veronese embedding $\nu:\p^1\times\p^1\to\SV_{1\times 1}^{1,d-1}\subset\p^{2d-1}$. Let $r\in\{2,\ldots,\rd{\frac{2d}{3}}\}$ and let $S\subset\p^1\times\p^1$ be a set of $r$ points. If $f(S)+\{q_1\}$ admits an unexpected curve in degree $d-r$, then $\nu(S)\in\Ter_r(\SV_{1\times 1}^{1,d-1})$.
\begin{proof}
By \Cref{lemma: multiproj-affine-proj}, we have
\[I_{2S,\p^1\times\p^1}(1,d-1)\cong I_{2f(S)+q_1+(d-1)q_2,\p^2}(d)\cong I_{f(S)+q_1+(d-1-r)q_2,\p^2}(d),\]
where the second isomorphism comes from Bézout's theorem. Notice that these linear systems have not only the same $h^0$, but also the same $h^1$, so one is special if and only if the other one is. By hypothesis, the system on the right hand side has dimension larger than expected, so $I_{2S,\p^1\times\p^1}(1,d-1)$ is special as well and therefore $\nu(S)\in\Ter_r(\SV_{1\times 1}^{1,d-1})$.
\end{proof}
\end{proposition}

We believe that this connection can be useful, because unexpected curves have been intensively studied. Just to recall a few results,
\begin{itemize}
    \item \cite[Lemma 3.5(c)]{unexpected} bounds the minimal degree of an unexpected curve.
\item \cite[Theorem 1.2 and Corollary 6.8]{unexpected}
show that if a set of $r+1$ points admits an unexpected curve in degree $d-r$ then no $d-r+1$ are collinear, but at least three of them are collinear.
\item \cite[Theorem 3.7]{unexpected} studies Hilbert functions of sets admitting unexpected curves.
\item \cite[Theorems 1.4 and 1.6]{unexpectedquartic} and \cite[Theorem 4.4]{dimca} classify low degree unexpected curves.
\end{itemize}

\begin{example} Let $\SV_{1\times 1}^{1,11}\subset \p^{23}$ and consider $r=8$. Let $S\subset\p^1\times\p^1$ be a set of eight points and take two general points $q_1,q_2\in\p^2$. By Lemma \ref{lemma: multiproj-affine-proj},
\begin{align*}
h^0(I_{2S,\p^1\times\p^1}(1,11))=h^0(I_{2S+11q_1+q_2,\p^2}(12))=h^0(I_{S+3q_1+q_2,\p^2}(4)).
\end{align*}
There are two possibilities in order for $S+3q_1+q_2$ to fail to give independent conditions on plane quartics. The first possibility is that $h^1(I_{S+q_2}(4))>0$. This happens for instance if six of the nine points are on a line, and gives some component of $\Ter_8(\SV_{1\times 1}^{1,11})$. The second possibility is that $S+q_2$ gives independent conditions on quartics but it admits an unexpected curve in degree 4. By \cite[Theorem 1.6]{unexpectedquartic}, up to projectivity there is only one such configuration of points, so this gives another component of $\Ter_8(\SV_{1\times 1}^{1,11})$ of dimension $\dim \PGL(2)=8$. 
\end{example}

\section{Algorithms}\label{section: algorithms}

In this section we present algorithms for finding the ideal of a multiprojective variety with a dimension that matches that of $\Ter_r(X)$ for some projective variety $X$ given a parametrization or an ideal. Let $R=\kk[x_0,\ldots,x_n]$.  Given a matrix $A$ with entries in $R$ and a set $S=\{p_1,\ldots,p_r\}$ of points in $\PP^n$, we can construct the ``stacked'' matrix $A(S)$ by evaluating $A$ at each $p_i$ and concatenating the rows, i.e.,

\begin{equation*}
A(S) = \begin{pmatrix} A(p_1) \\
\vdots \\
A(p_r)
\end{pmatrix}
\end{equation*}
for some choice of the homogeneous coordinates of the elements of $S$.  Since the determinant is alternating and multilinear, the rank of $A(S)$ does not depend on the order or the choice of homogeneous coordinates of the elements of $S$. Given a tuple $\mathbf f=(f_0,\ldots,f_m)\in R^{m+1}$, will we denote by $J_{\mathbf f}$ its Jacobian, i.e., $(J_{\mathbf f})_{ij}=\begin{pmatrix} \frac{\partial f_j}{\partial x_i}\end{pmatrix}$.

We will deal with two distinct settings: one where the tuple $\mathbf f$ parametrizes $X$ and so the ambient space has dimension $m$
and another where $\mathbf f$ generates $I(X)$ and so the ambient space has dimension $n$% from the number of variables in the ring $R$
.  However, in this latter case, we can only say something about the $2$-Terracini locus.  Note that in the following two lemmas, the choice to express the upper bound on the rank of the stacked matrix in terms of the minimum of two values, even though the minimum is clear, is made to simplify generalizing both settings to \cref{lemma: construction of ideal of terracini locus} below.

\begin{lemma}
\label{lemma: stacked jacobian for parametrizations}
  Suppose $X\subset\PP^m$ is parametrized by the rational map $\varphi:\PP^n\dashrightarrow\PP^m$ defined by $\mathbf f$.  If $S\subset\PP^n$ and $\varphi(S)$ consists of $r$ smooth points of $X$, then $\varphi(S)\in\Ter_r(X)$ if and only if  $J_{\mathbf f}(S)$ has rank less than $\min\{r(\dim X+1), m+1\}$.
\end{lemma}

\begin{proof}
  By construction, the projectivization of the row space of $J_{\mathbf f}(S)$ is $\langle T_{\varphi(p_1)}X,\ldots, T_{\varphi(p_r)}X\rangle$, and so the result follows by \Cref{def: Terracini locus with tangent spaces} and \eqref{eq: subabundant}.
\end{proof}

\begin{lemma}
\label{lemma: stacked jacobian for ideals}
  Suppose $X\subset\PP^n$ is a variety with $I(X)$ generated by $\mathbf f$.  If  $S$ consists of two smooth points of $X$, then $S\in\Ter_2(X)$ if and only if $J_{\mathbf f}^T(S)$ has rank less than $\min\{2\codim X, n+1\}$.
\end{lemma}

\begin{proof}
  By construction, the projectivization of the null space of $J_{\mathbf f}^T(S)$ is $ T_{p_1}X\cap T_{p_2}X$, and so by Grassmann's formula and the rank-nullity theorem, we have
  \begin{equation}
    \dim\langle  T_{p_1}X, T_{p_2}X\rangle = 2(\dim X + 1)-(n + 1 - \rank J_{\mathbf f}^T(S)) - 1.
  \end{equation}
  By \Cref{def: Terracini locus with tangent spaces}, we see that $S\in\Ter_2(X)$ precisely when $\rank J_{\mathbf f}^T(S)< n +1$.  Also, by \eqref{eq: subabundant}, $2(n - \codim X + 1)\leq n + 1$, which implies that $n + 1\leq 2\codim X$.
\end{proof}

Recall that $X^{(r)}$ is the quotient of the ordinary $r$-fold product $X^r$ by the group action of the symmetric group $\mathfrak S_r$, and thus there exists a natural finite map $\pi:X^r\to X^{(r)}$.  In particular, the ``unsymmetrized'' Terracini locus $\pi^{-1}(\Ter_r(X))$ has the same dimension as $\Ter_r(X)$.  In the case where $X$ is the image of a rational map as in Lemma \ref{lemma: stacked jacobian for parametrizations}, we can further pull back to $(\PP^n)^r$ and examine $(\varphi\times\cdots\times\varphi)^{-1}(\pi^{-1}(\Ter_r(X)))$.

Our goal is to compute the ideal of the unsymmetrized Terracini locus in the homogeneous coordinate ring of $(\PP^n)^r$, that is, $Q=\kk[z_{00},\ldots,z_{(r - 1)n}]$ with the standard $\ZZ^r$-grading ($\deg z_{ij}=(0,\ldots,1,\ldots,0)$ with the 1 in the $(i+1)$th entry). Then the dimension of $\Ter_r(X)$ will be the Krull dimension of this ideal minus $r$ in the case of \cref{lemma: stacked jacobian for ideals}.  In the case of \cref{lemma: stacked jacobian for parametrizations}, this gives us a lower bound on the dimension, since we may lose components involving the exceptional set when pulling back to $(\PP^n)^r$.

Given an $s\times t$ matrix $A$ with entries in $R$, let $A(z_{i\bullet})$ be the matrix be obtained by substituting $z_{ij}$ for $x_j$ in $A$ and $A(z_{\bullet\bullet})$ the stacked matrix obtained by concatenating the rows of each $A(z_{i\bullet})$. For all $0\leq i<j\leq r - 1$, let $Z_{ij}=\begin{pmatrix} z_{i0} & \ldots & z_{in} \\ z_{j0} & \ldots & z_{jn}\end{pmatrix}$.  For a given positive integer $k$ and matrix $B$ with entries in $Q$, let $\Sigma_k(B)$ be the determinantal ideal generated by the $k\times k$ minors of $B$.  For a given ideal $I$ of $R$, let $\ell=\rank A$ if $I=(0)$ and $\ell=\codim I$ otherwise, let $I(z_{i\bullet})$ be the ideal of $Q$ generated by the generators of $I$ but substituting $z_{ij}$ for $x_j$, and let $I(z_{\bullet\bullet})=\sum_i I(z_{i\bullet})$. 
 Finally, define the ideal 
  \begin{equation*}
  T_{A, I} =
    \sqrt{\left(\left(\Sigma_{\min\{r\ell, t\}}(A(z_{\bullet\bullet})) + I(z_{\bullet\bullet})\right):\left(\bigcap_{i,j}\Sigma_2(Z_{ij})\right)^\infty\right):\left(\bigcap_{i}\Sigma_{\ell}(A(z_{i\bullet}))\right)^\infty}
  \end{equation*}
  of $Q$. Note that $\Sigma_2(Z_{ij})$ is the ideal of duplicate points in the $i$th and $j$th components of $(\PP^n)^r$ and $\Sigma_{\ell}(A(z_{i\bullet}))$ is the ideal of points in the $i$th component where $A$ has rank less than $\ell$.  By saturating out these ideals and applying the Nullstellensatz, we obtain the following result.

  \begin{lemma}
    \label{lemma: construction of ideal of terracini locus}
    $T_{A,I}$ is the ideal of all tuples $(p_1,\ldots,p_r)$ of distinct points in $V(I)^r$ for which  $A(p_i)$ has rank at least $\ell$ but $\rank A(\{p_1,\ldots,p_r\})<\min\{r\ell, t\}$.
  \end{lemma}

\Cref{lemma: construction of ideal of terracini locus} generalizes both of our settings.  When $X\subset\PP^m$ is parametrized by $\mathbf f$, then we use $A=J_{\mathbf f}$ and $I=(0)$.  In this case, $\ell=\rank J_{\mathbf f}=\dim X 
 + 1$, and by applying \Cref{lemma: stacked jacobian for parametrizations}, we get the following result.

  \begin{proposition}
  If $X\subset\PP^m$ is parametrized by the rational map $\varphi:\PP^n\dashrightarrow\PP^m$ defined by $\mathbf f$, then $\dim\Ter_r(X)\geq\dim T_{J_{\mathbf f},(0)} - r$.
  \end{proposition}

  When $X\subset\PP^n$ is the vanishing locus of the ideal generated by $\mathbf f$, we use $A=J_{\mathbf f}^T$ and $I=(\mathbf f)$.  In this case, $\ell=\codim I=\codim X$, and we get the following result by applying \Cref{lemma: stacked jacobian for ideals}.

  \begin{proposition}
    If $X\subset\PP^n$ is a variety with $I(X)=(\mathbf f)$, 
    then $\dim\Ter_2(X)=\dim T_{J_{\mathbf f}^T,(\mathbf f)} - 2$.
  \end{proposition}

These results are implemented in a \textit{TerraciniLoci} package for the computer algebra system \textit{Macaulay2} \cite{M2}.  The \textit{CorrespondenceScrolls} package \cite{CorrespondenceScrolls} is used to construct the multigraded ring $Q$, and the \textit{FastMinors} package \cite{FastMinors} is used the speed up the minor computations.  See \url{https://github.com/d-torrance/terracini-loci} for the code and examples of many of the results of this paper, including Terracini loci of some rational and elliptic normal curves supporting \cref{corol: rational and elliptic normal curves}, the rational octic in $\PP^4$ mentioned after \cref{coroll: curves embedded with non-complete linear systems}, the rational quintic in $\PP^4$ from \cref{example: rational quintic in P4 without Terracini points}, the del Pezzo surfaces from \cref{coroll: Ter2 per delpezzos} and \cref{cor: ter3 of del pezzo}, the Veronese cubic surface supporting \cref{propo: all we know about the Veronese}, and several Segre-Veronese surfaces supporting \cref{theorem: first nonempty SV}, including the del Pezzo surface not addressed in \cref{section: del Pezzo}.
Note that despite the assumption that $\kk$ be algebraically closed of characteristic 0, we use $\mathbb Q$ or $\ZZ/p\ZZ$ for moderately large prime $p$ to avoid floating point errors.

\section*{Acknowledgements}
This project started during the semester program Algebraic Geometry with Applications to Tensors and Secants (AGATES), held in Warsaw in 2022, where Luca Chiantini proposed Terracini loci as a research topic. We wish to thank Jarosław Buczyński, Fulvio Gesmundo, Alex Massarenti, Tim Seyennaeve, Trygve Johnsen and Mahrud Sayrafi for their support and for helpful discussion.\\
This work is partially supported by the Thematic Research Programme ``Tensors: geometry, complexity and quantum entanglement'', University of Warsaw, Excellence Initiative - Research University and the Simons Foundation Award No. 663281 granted to the Institute of Mathematics of the Polish Academy of Sciences for the years 2021-2023.\\
Galuppi acknowledges support by the National Science Center, Poland, project ``Complex contact manifolds and geometry of secants'', 2017/26/E/ST1/00231. Santarsiero is supported by the Deutsche Forschungsgemeinschaft (DFG, German Research Foundation) -- Projektnummer 445466444. Turatti is partially supported by the project Pure Mathematics in Norway funded by Trond Mohn Foundation and by Tromsø Research Foundation grant agreement 17matteCR.

\bibliographystyle{alpha}
\bibliography{References_GSTT.bib}
\end{document}